\documentclass[a4paper,11pt]{amsart}
\usepackage[left=2.5cm, right=2.5cm,top=2.5cm,bottom=3cm]{geometry}

\usepackage[utf8]{inputenc}
\usepackage[english]{babel}
\usepackage{amssymb}
\usepackage{amsfonts}
\usepackage{dsfont}
\usepackage{csquotes}

\usepackage{hyperref}

\usepackage{tikz}
\usepackage{chemarr}

\usepackage{graphicx}
\graphicspath{ {./} }
\usepackage{hhline}
\usepackage[sort&compress,comma,square,numbers]{natbib}

\DeclareMathOperator{\LC}{LC}

\DeclareMathOperator{\Log}{Log}

\DeclareMathOperator{\N}{NP}
\DeclareMathOperator{\inte}{int}
\DeclareMathOperator{\relint}{relint}
\DeclareMathOperator{\Conv}{Conv}

\DeclareMathOperator{\val}{val}

\DeclareMathOperator{\Trop}{Trop}

\newcommand{\R}{\mathbb{R}}

\newtheorem{thm}{Theorem}[section] 
\newtheorem{cor}[thm]{Corollary} 
\newtheorem{prop}[thm]{Proposition} 
\newtheorem{lemma}[thm]{Lemma}

\theoremstyle{definition}

\newtheorem{ex}[thm]{Example}
\newtheorem{remark}[thm]{Remark}

\begin{document}

\title{Real tropicalization and negative faces of the Newton polytope}

\author{Máté L. Telek}
\address{Department of Mathematical Sciences, University of Copenhagen,
Universitetsparken 5,
2100 Copenhagen, Denmark}
\email{mlt@math.ku.dk}

\begin{abstract}
In this work, we explore the relation between the tropicalization of a real semi-algebraic set $S = \{ f_1 < 0, \dots , f_k < 0\}$ defined in the positive orthant and the combinatorial properties of the defining polynomials $f_1, \dots, f_k$. We describe a cone that depends only on the face structure of the Newton polytopes of $f_1, \dots ,f_k$ and the signs attained by these polynomials. This cone provides an inner approximation of the real tropicalization, and it coincides with the real tropicalization if $S = \{ f < 0\}$ and the polynomial $f$ has generic coefficients. Furthermore, we show that for a maximally sparse polynomial $f$ the real tropicalization of $S = \{ f < 0\}$ is determined by the outer normal cones of the Newton polytope of $f$ and the signs of its coefficients. Our arguments are valid also for signomials, that is, polynomials with real exponents defined in the positive orthant.

\end{abstract}

\maketitle


\section{Introduction}

Real algebraic varieties, or more generally, semi-algebraic sets, play a central role in applications, e.g. in chemical reaction network theory \cite{CRN_Dickenstein} or in robotics \cite{Robotics}. In practice, the polynomials defining these sets involve many variables and many monomials, which makes using standard techniques from real (semi-)algebraic geometry infeasible.

Tropicalization is the process of associating a polyhedral complex to an algebraic variety. This can be used to answer questions about the variety, such as computing its dimension \cite[Structure Theorem]{maclagan2015introduction}, or to compute Gromov-Witten invariants, that is, to count the number of curves of given degree and genus passing through a given finite number of points in the complex projective plane \cite{Mikhalkin2003EnumerativeTA}. Moreover, using tropicalizations one can find the generic number of solutions of a parametric polynomial equation system \cite{YuePaul}.

One of the roots of tropical geometry goes back to 1971 when Bergman \cite{Bergman_LogLimit} studied the logarithmic limit of an algebraic variety $V$ in the complex torus $(\mathbb{C}^*)^{n}$. The logarithmic limit of $V$ is defined as the limit in the Hausdorff distance of the images of $V$ under the coordinate-wise logarithm map with base $t > 0$
\[ \Log_t \colon (\mathbb{C}^{*})^n \mapsto \mathbb{R}^{n}, \quad z \mapsto (\log_t( \lvert z_1 \rvert), \dots , \log_t( \lvert z_n \rvert)\]
as $t \to \infty$. If $V$ is a hypersurface, i.e. it is defined by one polynomial $f$, the logarithmic limit of $V$ equals the $(n-1)$-skeleton of the outer normal fan of the Newton polytope of $f$ \cite[Corollary 6.4]{MIKHALKIN20041035}.

The logarithmic limit of a real semi-algebraic set $S \subseteq \mathbb{R}_{>0}^{n}$ was first studied by Alessandrini \cite{Alessandrini2007LogarithmicLS} in 2013. The author showed that the logarithmic limit of $S$ is a closed polyhedral complex of dimension at most the dimension of $S$ \cite[Theorem 3.11]{Alessandrini2007LogarithmicLS}. In \cite{RealTrop2022, BLEKHERMAN2022108561}, the authors called the logarithmic limit of $S$ the \emph{real tropicalization} of $S$. In this manuscript, we follow this notation and write
\begin{align*}
\Trop(S) := \lim_{t \to \infty} \Log_{t}(S).
\end{align*}

Computing the real tropicalization of a semi-algebraic set $S \subseteq \mathbb{R}_{>0}^{n}$ is known to be hard. The Fundamental Theorem  \cite[Theorem 6.9]{RealTrop_SemiAlgSet} implies that $\Trop(S)$ is described as an intersection of outer normal cones as follows: For a polynomial $f \in \mathbb{R}[x_1, \dots ,x_n]$, we call an exponent vector negative if the coefficient of the corresponding monomial is negative. Similarly, a face $F$ of the Newton polytope $\N(f)$ is a \emph{negative face} if $F$ contains a negative exponent vector. The union of the outer normal cones of $\N(f)$ that correspond to negative faces is denoted by $\mathcal{N}_f^-$. Using this notation, we have: 
\begin{align}
\label{Eg_TropDefinition}
 \Trop(S) = \bigcap_{f \leq 0 \text{ on } S} \mathcal{N}_f^-,
\end{align}
where the intersection is taken over all polynomials $f \in \mathbb{R}[x_1, \dots ,x_n]$ which are nonpositive on $S$.

In \cite[Theorem 6.9]{RealTrop_SemiAlgSet}, it was also shown that the intersection in \eqref{Eg_TropDefinition} can be taken over finitely many such polynomials, but to the best of our knowledge, no algorithm is known to find such a finite set of polynomials. However, in certain cases, there exist good approximations of the real tropicalization. From \cite[Proposition 6.12]{RealTrop_SemiAlgSet}, it follows that the real tropicalization of a semi-algebraic set $S$ of the form $S = \bigcap_{i=1}^{k} f_i^{-1} (\mathbb{R}_{< 0})$ satisfies
\begin{align}
\label{Eq_InclPoly}
\inte\big( \bigcap_{i=1}^{k} \mathcal{N}_{f_i}^- \big) \subseteq \Trop(S) \subseteq \bigcap_{i=1}^{k} \mathcal{N}_{f_i}^-.
\end{align}
If the set on the right-hand side in \eqref{Eq_InclPoly} is \emph{regular} (it is equal to the closure of its interior), then the real tropicalization of $S$ equals the intersection of the cones $\mathcal{N}_{f_i}^-, i = 1, \dots ,k$ \cite[Corollary 4.8]{TropSpecta}. In Proposition \ref{Lemma_CommonNegVertex}, we characterize when regularity happens. As a consequence of this, we show that if $S$ is defined by one polynomial $f$, i.e. $S= f^{-1}(\mathbb{R}_{<0})$, then $\mathcal{N}_{f}^-$ is regular if and only if the faces of $\N(f)$ that contain a negative exponent vector have also a negative vertex. This seems to be a quite restrictive assumption, but it is automatically satisfied for polynomials whose set of exponent vectors equals the set of vertices of the Newton polytope. Such polynomials are called \emph{maximally sparse} in the literature \cite{Nisse2008maximally}. Thus, if $f$ is maximally sparse, then the real tropicalization of $S= f^{-1}(\mathbb{R}_{<0})$ is easy to compute as it equals $\mathcal{N}_f^-$ (Corollary \ref{Cor_MaxSparse}).

Our main result is Theorem \ref{Thm_Inclusions}, where we give an elementary proof of the inclusions in \eqref{Eq_InclPoly}, while we also refine them. We associate a cone $\Sigma(f_1, \dots , f_k)$, called the \emph{actual negative normal cone} of $f_1, \dots , f_k$ (see Section \ref{Sec_Prelim}) that provides a better inner approximation of $\Trop(S)$:
\begin{align}
\label{Eq_InclPoly2}
\inte\big( \bigcap_{i=1}^{k} \mathcal{N}_{f_i}^- \big) \subseteq \Sigma(f_1, \dots , f_k)  \subseteq \Trop(S) \subseteq \bigcap_{i=1}^{k} \mathcal{N}_{f_i}^-.
\end{align}
The proof of Theorem \ref{Thm_Inclusions} is valid also in the case where the defining polynomials of the semi-algebraic set have real exponents. 

Theorem \ref{Thm_Inclusions} has two main consequences. In the case when $S$ is defined by one polynomial, i.e. $S= f^{-1}(\mathbb{R}_{<0})$, we show that $\Sigma(f)$ equals the real tropicalization of $S$ when the coefficients of $f$ are generic (Corollary \ref{Cor_RealTropOnePoly}). 

Specifically, to compute $\Sigma(f)$, one needs to determine all faces of the Newton polytope of $f$ for which the restriction of $f$ to the face takes negative values over $\mathbb{R}_{>0}^{n}$. Deciding if a polynomial is nonnegative is an NP-hard problem \cite{Laurent2009}. To address this difficulty, several sufficient criteria implying nonnegativity have been introduced. For instance, one approach, going back to Hilbert \cite{Hilbert1888berDD}, is to express the polynomial as a sum of squares (SOS); this decomposition can be found by semidefinite programming, see  \cite{Laurent2009} and the references therein. Another approach is to write the polynomial as sums of nonnegative circuits (SONC) \cite{SONC,DRESSLER2019149}, see also \cite{Reznick1989FormsDF,Chandrasekaran2014RelativeER}.

The set of all nonnegative polynomials is a closed convex cone, called the nonnegativity cone. The sets of SOS and SONC polynomials are also convex cones, clearly contained in the nonnegativity cone. The genericity assumption in Corollary \ref{Cor_RealTropOnePoly} is that neither $f$ nor its restrictions to the faces of the Newton polytope lie on the boundary of the nonnegativity cone. To the best of our knowledge, there are no complete descriptions of the boundary neither of the nonnegativity cone nor of the SOS cone. On the contrary, the boundary of the SONC cone has been characterized in \cite{SoncBoundary}. This characterization together with Corollary \ref{Cor_RealTropOnePoly} provide a method to compute the real tropicalization of semi-algebraic sets defined by one polynomial. 

In \cite{NonNegCone}, the authors introduced the DSONC cone, which is a cone contained in the interior of the SONC cone. A remarkable property of the DSONC cone is that it is possible to check membership using linear programming. This might give a more efficient method to compute $\Sigma(f)$ and to verify the genericity condition in Corollary \ref{Cor_RealTropOnePoly}.

Finally, an additional consequence of Theorem \ref{Thm_Inclusions}, which is an interesting result on its own, is that any logarithmic image of $f^{-1}(\mathbb{R}_{<0})$ is a bounded set if the boundary of $\N(f)$ does not contain any negative exponent vector (Corollary \ref{Cor_BoundedNegPre}).

\subsection*{Notation} $\mathbb{R}_{\geq0}$, $\mathbb{R}_{>0}$ and $\R_{<0}$ refer to the sets of nonnegative, positive and negative real numbers respectively. For two vectors $v,w \in \mathbb{R}^{n}$, $v \cdot w$ denotes the Euclidean scalar product, and $v \ast w$ denotes the coordinate-wise product of $v$ and $w$. We denote the interior of a set $X \subseteq \mathbb{R}^{n}$ by $\inte(X)$. If $X \subseteq \mathbb{R}^{n}$ is a polyhedron, $\relint(X)$ denotes the relative interior of $X$. The symbol $\#S$ denotes the cardinality of a finite set $S$.

\section{Negative faces of the Newton polytope}
\label{Sec_Prelim}

 Following \cite{duffin1973geometric, FuzzyGeo}, a \emph{signomial} is a function
\begin{align*}
f\colon \mathbb{R}^{n}_{>0} \to \mathbb{R}, \qquad x \mapsto f(x) = \sum_{\mu \in \sigma(f)} c_{\mu} x^{\mu},
\end{align*}
where $\sigma(f) \subseteq \mathbb{R}^{n}$ is a finite set, called the \emph{support} of $f$, and the \emph{coefficients} $c_{\mu} \in \mathbb{R}$ are non-zero. We use the notation $x^{\mu}$ for $x_1^{\mu_1} \cdots x_n^{\mu_n}$. Thus, a signomial is a polynomial with real exponents whose domain is restricted to the positive orthant. For a set $S \subseteq \mathbb{R}^{n}$, we define the \emph{restriction} of $f$ to $S$ as
\begin{align*}
f_{|S}(x) := \sum_{\mu \in \sigma(f) \cap S} c_{\mu} x^{\mu}.
\end{align*}
Furthermore, we divide the support of $f$ into the set of\emph{ positive and negative exponent vectors}:
\[ \sigma_+(f) := \{ \mu \in \sigma(f) \mid c_{\mu} > 0 \},  \qquad \sigma_-(f) := \{ \mu \in \sigma(f) \mid c_{\mu} < 0 \}. \]
Thus, a positive exponent vector $\mu$ is an exponent vector of a monomial $x^{\mu}$ of $f$ such that the corresponding coefficient $c_\mu$ is positive.

\medskip

The \emph{Newton polytope} of $f$ is the convex hull of the support
\begin{align*}
\N(f) := \Conv\big(\sigma(f)\big).
\end{align*}
We refer to Figure~\ref{FIG1}(a) for an illustration of the Newton polytope of $f = x_1^2-x_1+1 - x_2^2$. The set of negative exponent vectors is $\sigma_-(f) = \{(1,0),(0,2)\}$ and the set of positive exponent vectors is $\sigma_+(f) =  \{(0,0),(2,0)\}$.

We recall some basic notions in polyhedral geometry that are relevant to the study of real tropicalizations (see  \cite{Ziegler_book, BasicsOnPolytopes, JoswigTheobald_book} for details). Each vector $v \in \mathbb{R}^{n}$ cuts out the \emph{face} of $\N(f)$ \cite[Section 15.1.1]{BasicsOnPolytopes} given by
\begin{align}
\label{Eq_vface}
\N(f)_v := \{ \mu \in \N(f) \mid v \cdot \mu = \max_{\nu \in \N(f)} v \cdot \nu\}.
\end{align}
The vector $v$ is called an \emph{outer normal vector} of $\N(f)_v$, or normal vector for short. A face is \emph{proper} if it is not the entire polytope.

For a face $F \subseteq \N(f)$, the set of all vectors $v$ such that $F \subseteq \N(f)_v$ form a closed convex polyhedral cone, called the \emph{normal cone} of $F$ \cite[Section 15.2.2]{BasicsOnPolytopes}:
\begin{align}
\label{Eq_DefNormalCone}
\mathcal{N}_{f}(F) = \{ v \in \mathbb{R}^{n} \mid F \subseteq \N(f)_v \} .
\end{align}
Furthermore, for $v \in \mathcal{N}_f(F)$, it holds:
\begin{align}
\label{Eq_BdCorr}
\N(f)_v = F \quad &\Longleftrightarrow \quad v \in \relint \mathcal{N}_f(F).
\end{align}
The collection of the normal cones of all faces forms a \emph{complete fan} $\mathcal{N}_f$ \cite[Example 7.3]{Ziegler_book}.

Taking normal cones induces an inclusion reversing correspondence between faces and their normal cones, which behaves well with dimensions. Hence, for any two non-empty faces $F,G \subseteq \N(f)$, it holds that:
\begin{align}
\label{Eq_FaceConeCorr}
F \subseteq G \quad &\Longleftrightarrow \quad \mathcal{N}_f(F) \supseteq \mathcal{N}_f(G), \\
\label{Eq_DimCorr}
\dim F &= n - \dim \mathcal{N}_f(F).
\end{align}

To study several signomials $f_1, \dots, f_k$ at the same time, it will be beneficial to consider the \emph{common refinement} of the normal fans $\mathcal{F}_{f_1}, \dots , \mathcal{F}_{f_k}$\cite[Definition 7.6]{Ziegler_book}, which is the fan defined as
\begin{align}
\label{Eq_CommonRefinement}
\bigwedge_{i=1}^{k} \mathcal{N}_{f_i} := \Big\{ \bigcap_{i=1}^{k} C_i \mid C_i \in \mathcal{N}_{f_i} \Big\}.
\end{align}
It is easy to see from \eqref{Eq_CommonRefinement} that full-dimensional cones in the common refinement are intersections of full-dimensional cones. Note that for every pair of vectors $v,v'$ in the relative interior of a cone in the common refinement, it holds that
\begin{align}
\label{Eq_RelIntCommonRef}
\N(f_i)_v = \N(f_i)_{v'} \qquad \text{for all } i=1,\dots ,k.
\end{align}

\medskip

Several arguments in this work rely on the following easy observation. Each $v \in \mathbb{R}^{n}$ and $x \in \mathbb{R}^{n}_{>0}$ induce a signomial in one variable
\begin{align}
\label{Eq_InducedUniv}
\mathbb{R}_{>0} \to \mathbb{R},  \qquad t \mapsto f(t^v \ast x) = \sum_{\mu \in \sigma(f)} c_\mu x^\mu t^{v \cdot \mu},
\end{align}
where $t^v$ is short for the vector $(t^{v_1}, \dots , t^{v_n})$. We call a coefficient of a univariate signomial the \emph{leading coefficient} ($\LC$) if the exponent of the accompanying monomial is the largest. From \eqref{Eq_vface} it follows that for fixed $x \in \mathbb{R}^{n}_{>0}$
\[ \LC (f(t^v \ast x) ) =  \sum_{\mu \in \sigma(f) \cap \N(f)_v} c_\mu x^\mu,\]
if we view $f(t^v \ast x)$ as a univariate signomial in the variable $t$. This implies the following well-known fact (see e.g. \cite[Proposition 2.3]{MultDualPhos})
\begin{lemma}
\label{Lemma_Pushing}
Let $f\colon \mathbb{R}^{n}_{>0} \to \mathbb{R}$ be a signomial and $v \in \mathbb{R}^{n}$. If $f_{|N(f)_{v}}(x) < 0$ for some $x \in \mathbb{R}^{n}_{>0}$, then there exists $T \in \mathbb{R}_{>0}$ such that for all $t > T$:
\[ f(t^v \ast x) < 0 .\]
\end{lemma}

To find all vectors $v \in \mathbb{R}^{n}$ for which Lemma \ref{Lemma_Pushing} applies, we have to decide for which faces $F \subseteq \N(f)$, the signomial $f_{|F}$ takes negative values. A necessary condition for $f_{|F}$ to take negative values is that the face $F$ contains negative exponent vectors of $f$. Motivated by this observation, we call a face $F \subseteq \N(f)$ a \emph{negative face} if $F \cap \sigma_- (f) \neq \emptyset$. Furthermore, we define the \emph{negative normal cone} of $\N(f)$ as the union of the normal cones corresponding to the negative faces:
\[ \mathcal{N}_f^- =  \bigcup_{\begin{subarray}{c}F \subseteq \N(f) \\ \text{ negative face}\end{subarray}} \mathcal{N}_{f}(F) = \{ v \in \mathbb{R}^{n} \mid \N(f)_v \cap \sigma_-(f) \neq \emptyset \}.\]
Note that $\mathcal{N}_f^-$ is a cone, i.e. it is closed under multiplication by positive scalars, but $\mathcal{N}_f^-$ does not need to be convex.

Even if the face $F$ contains negative exponent vectors of $f$ there is no guarantee that $f_{|F}$ actually takes negative values. For an example, take the signomial $f = x_1^2-x_1+1 - x_2^2$ from Figure~\ref{FIG1} and the face $F = \Conv\big( (0,0) , (2,0) \big)$.

For a signomial $f$, we define the \emph{actual negative normal cone} as
\begin{align*}
\Sigma(f) := \Big\{ v \in \mathbb{R}^{n} \mid  f_{|\N(f)_v}^{-1}(\mathbb{R}_{<0} ) \neq \emptyset \Big\} \subseteq \mathcal{N}^-_f.
\end{align*}
In Proposition \ref{Prop_AcNegCone_closed}, we show that $\Sigma(f)$ is the union of normal cones corresponding to faces $F \subseteq \N(f)$ such that $f_{|F}$ takes negative values. The actual negative normal cone and the negative normal cone of  $f = x_1^2-x_1+1 - x_2^2$ are depicted in Figure~\ref{FIG1}(b)(c).

 Define the \emph{actual negative normal cone} of a collection of signomials $f_1, \dots , f_k$ as:
\begin{align}
\label{Eq_AcNegNormalCone}
\Sigma(f_1, \dots , f_k) := \Big\{ v \in \mathbb{R}^{n} \mid \bigcap_{i=1}^{k} f_{i|\N(f_i)_v}^{-1}(\mathbb{R}_{<0} ) \neq \emptyset \Big\}.
\end{align}

\begin{figure}[t]
\centering
\begin{minipage}[h]{0.3\textwidth}
\centering
\includegraphics[scale=0.4]{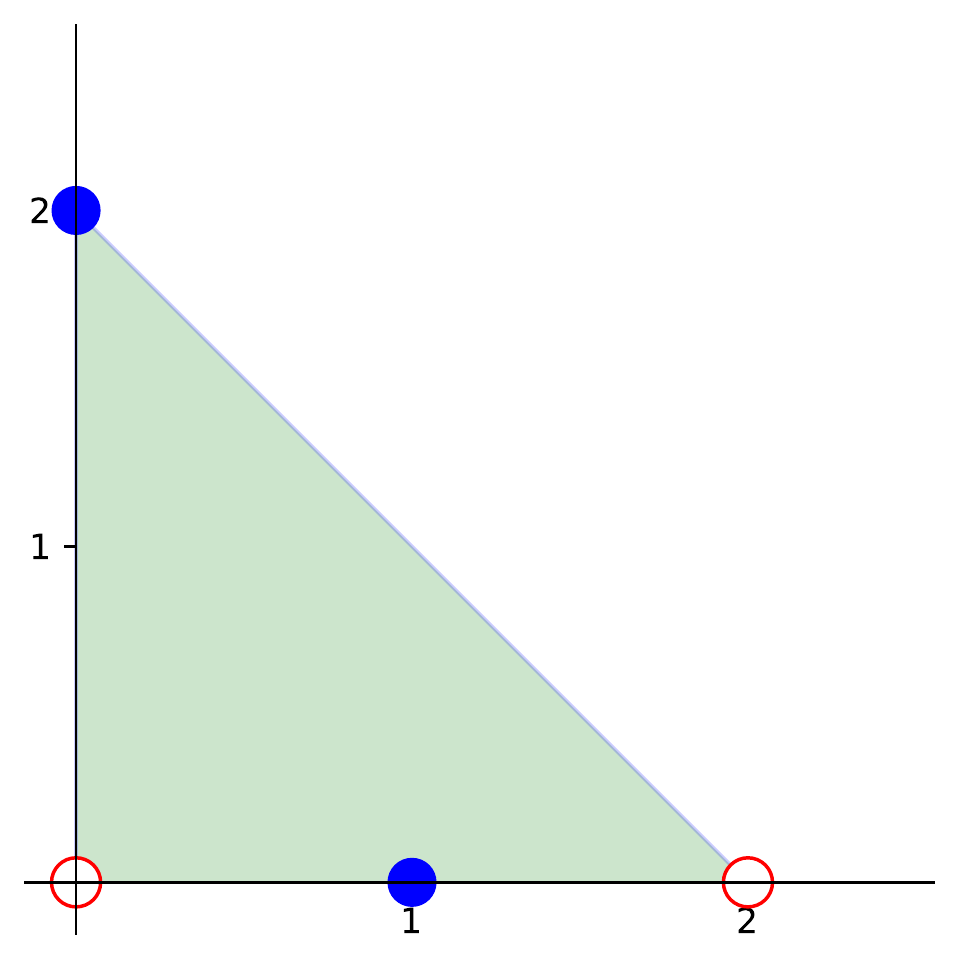}

{\small (a) $\N(f)$}
\end{minipage}
\begin{minipage}[h]{0.3\textwidth}
\centering
\includegraphics[scale=0.4]{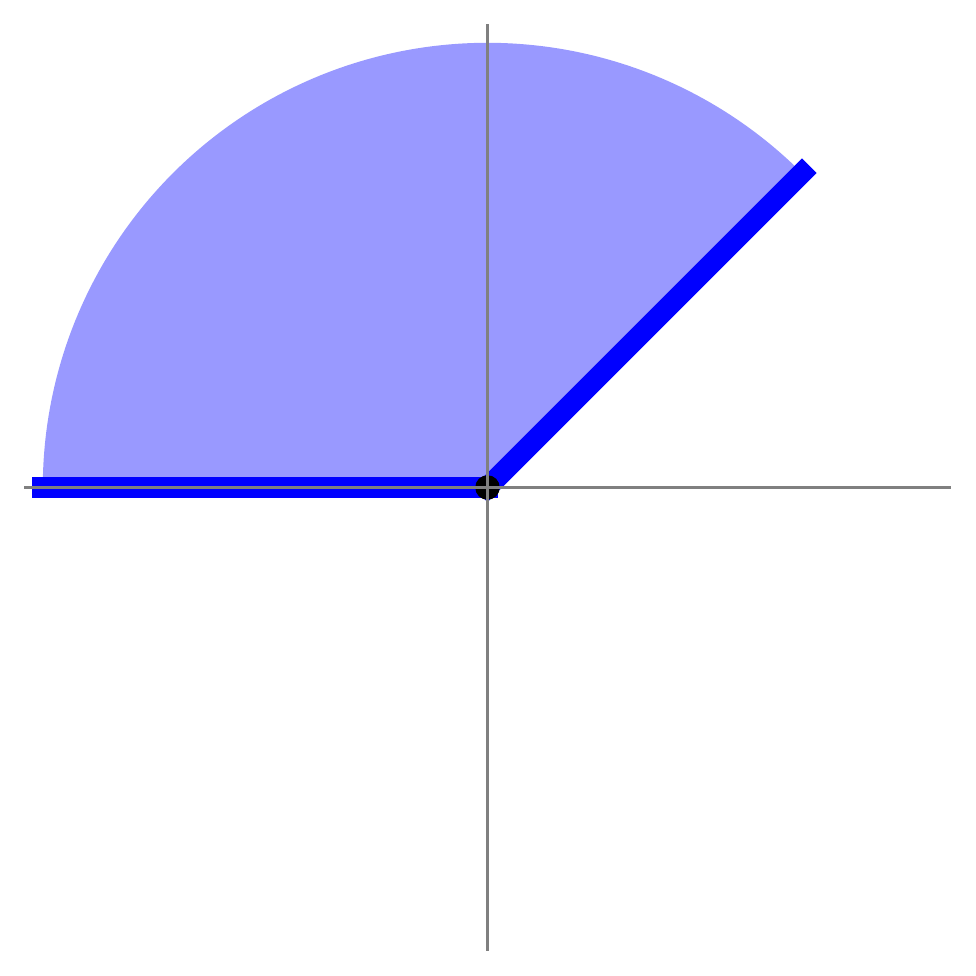}

{\small (b) $\Sigma(f)$ }
\end{minipage}
\begin{minipage}[h]{0.3\textwidth}
\centering
\includegraphics[scale=0.4]{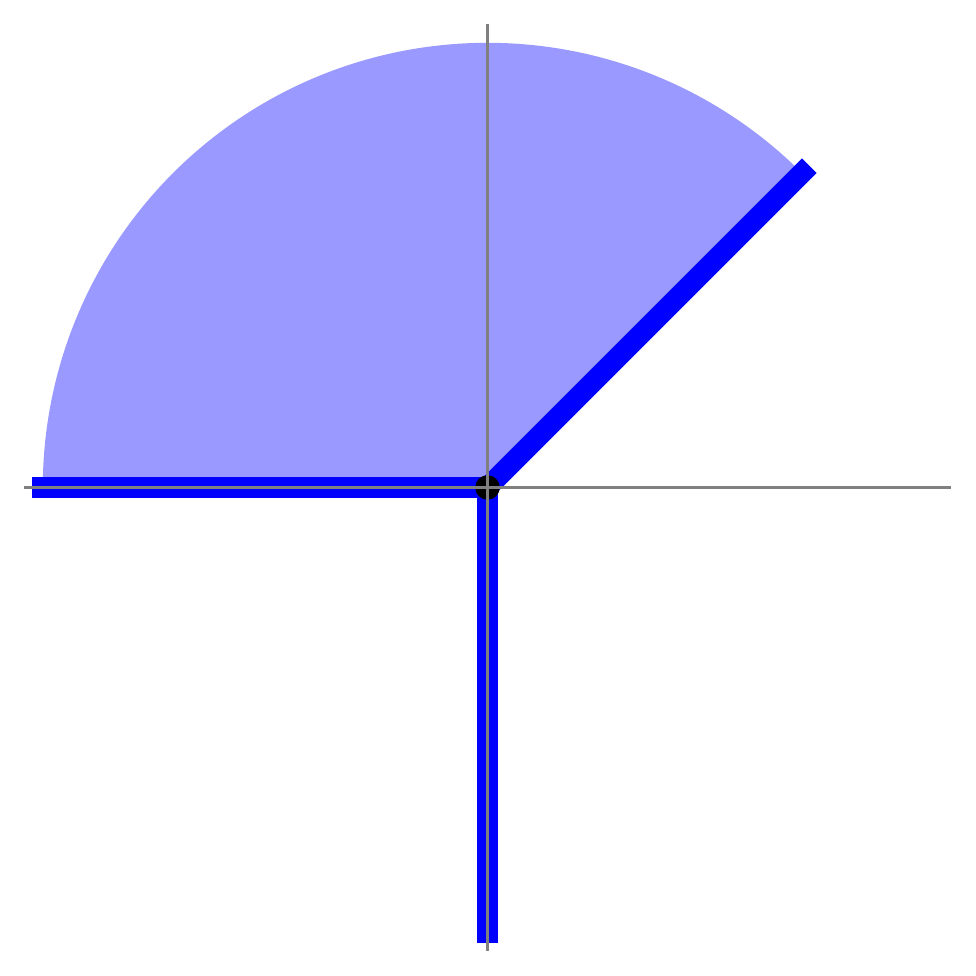}

{\small (c) $\mathcal{N}_f^-$}
\end{minipage}

\begin{minipage}[h]{0.3\textwidth}
\centering
\includegraphics[scale=0.4]{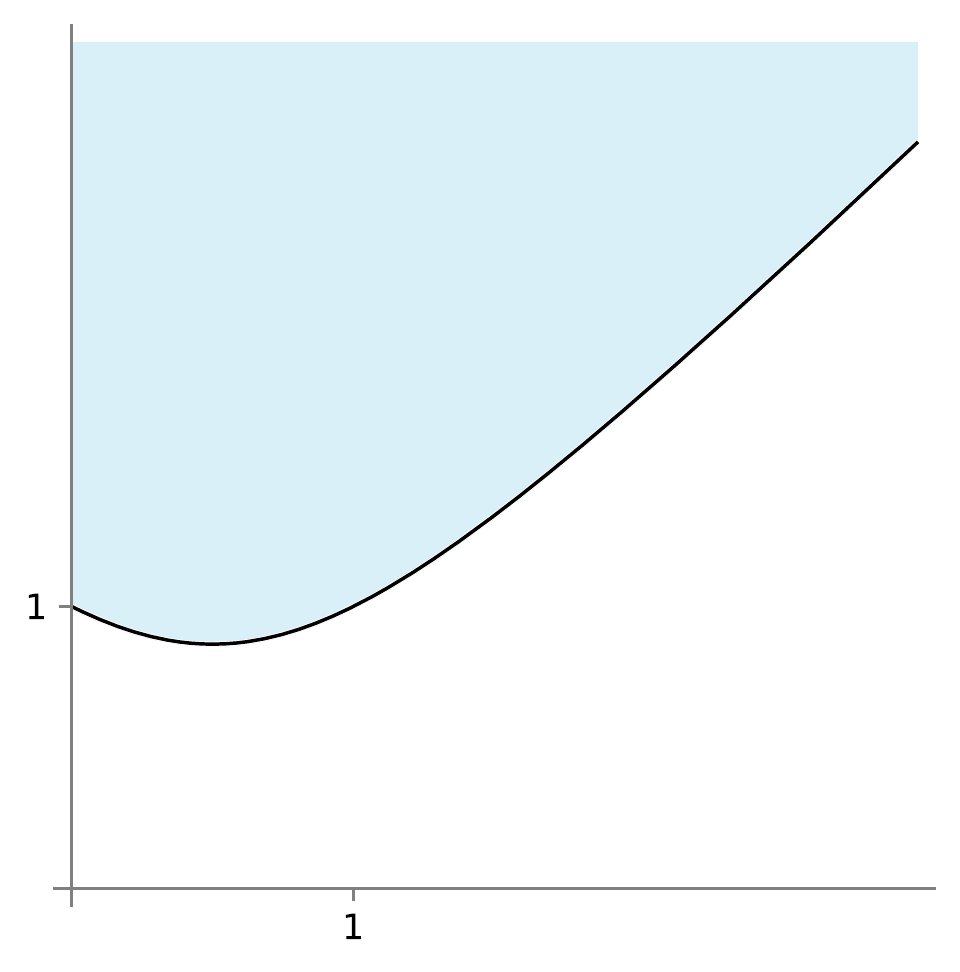}

{\small (d) $f^{-1}(\mathbb{R}_{<0})$ }
\end{minipage}
\begin{minipage}[h]{0.3\textwidth}
\centering
\includegraphics[scale=0.4]{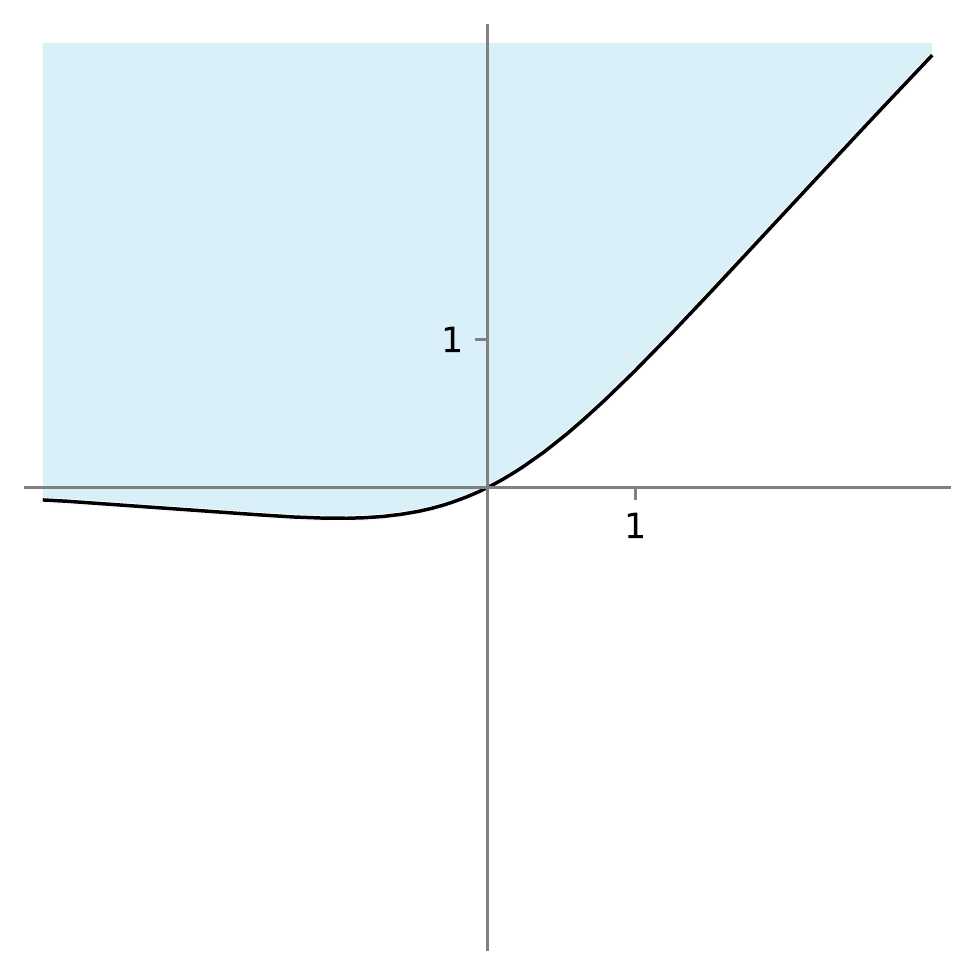}

{\small (e) $\Log_2\big(f^{-1}(\mathbb{R}_{<0})\big)$}
\end{minipage}
\begin{minipage}[h]{0.3\textwidth}
\centering
\includegraphics[scale=0.4]{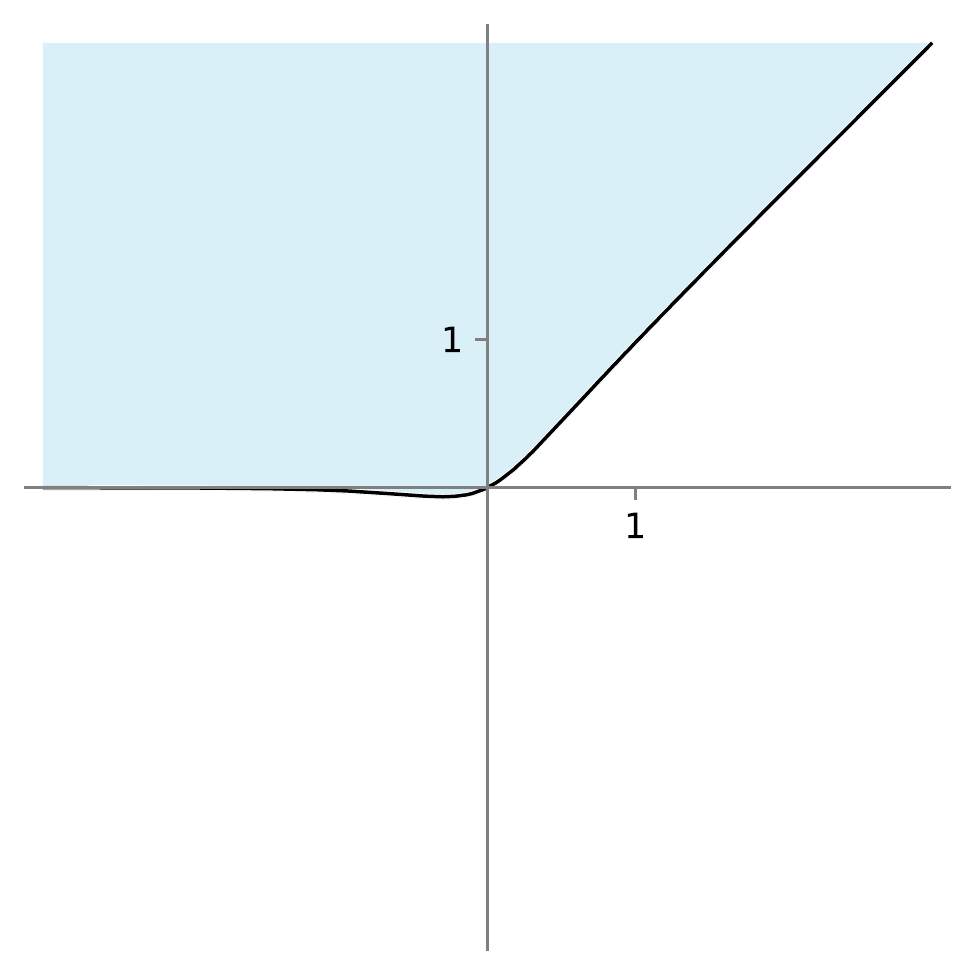}

{\small (f) $\Log_{10}\big(f^{-1}(\mathbb{R}_{<0})\big)$}
\end{minipage}
\caption{{\small (a) Newton polytope of $f = x_1^2-x_1+1 - x_2^2$. Blue dots correspond to negative exponents vectors, red circles correspond to positive exponent vectors. (b),(c) actual negative normal cone and negative normal cone of $f$, (d) semi-algebraic set defined by $f$, (e),(f) logarithmic images of $f^{-1}(\mathbb{R}_{<0})$.}}\label{FIG1}
\end{figure}

In the proof of Theorem \ref{Thm_Inclusions}, we need that the negative normal cones  $\mathcal{N}^-_{f_i},  i=1, \dots , k$ are closed subsets of $\mathbb{R}^{n}$. This property follows easily from the fact that a polytope has finitely many faces \cite[Theorem 3.46]{JoswigTheobald_book} and the cones  $\mathcal{N}_{f_i}(F)$, $F$ is a face of $\N(f_i)$, are closed. The actual negative normal cone $\Sigma(f_1, \dots ,f_k)$ is also closed, as the following argument shows.

\begin{prop}
\label{Prop_AcNegCone_closed}
Let $f_1, \dots , f_k$ be signomials on $\mathbb{R}^{n}_{>0}$, let $C$ be a cone in the common refinement $\bigwedge_{i=1}^{k} \mathcal{N}_{f_i}$ and $v \in \relint(C)$. If $v \in \Sigma(f_1, \dots ,f_k)$, then $C \subseteq \Sigma(f_1, \dots ,f_k)$. In particular, $\Sigma(f_1, \dots ,f_k)$ is a closed subset of $\mathbb{R}^n$.
\end{prop}

\begin{proof}
Let $C_i \in \mathcal{N}_{f_i}$, $i = 1 ,\dots ,k$ such that $C = \cap_{i=1}^{k} C_i$. Since $v \in \relint(C)$, from \cite[Theorem 6.5.]{Rockafellar} follows that $v \in \relint(C_i)$ for all $i = 1, \dots ,k$.

Let $F_i := \N(f_i)_v, i = 1, \dots , k$. Since $v \in \relint(C)$, for any $w \in C$ we have that 
\[ F_i  = \big(\N(f_i)_w \big)_v \subseteq \N(f_i)_w\]
for all $i = 1,\dots,k$ by \eqref{Eq_DefNormalCone} and \eqref{Eq_BdCorr}.

Consider $x \in\bigcap_{i=1}^{k} f_{i|F_i}^{-1}(\mathbb{R}_{<0} )$, which exists since $v \in \Sigma(f_1, \dots ,f_k)$. By Lemma \ref{Lemma_Pushing}, there exists $T > 0 $ such that
\[t^v \ast x \in \bigcap_{i=1}^{k} f_{i|\N(f_i)_w}^{-1}(\mathbb{R}_{<0} ), \qquad \text{for } t > T,\]
which implies that $w \in \Sigma(f_1, \dots ,f_k)$. Hence, $C \subseteq \Sigma(f_1, \dots ,f_k)$.

This argument shows that $\Sigma(f_1, \dots ,f_k)$ is the union of the cones in $\bigwedge_{i=1}^{k} \mathcal{N}_{f_i}$, whose relative interior intersects $\Sigma(f_1, \dots ,f_k)$. Since there exist only finitely many cones in $\bigwedge_{i=1}^{k} \mathcal{N}_{f_i}$ and they are closed, it follows that $\Sigma(f_1, \dots ,f_k)$ is a closed set.
\end{proof}

Proposition \ref{Prop_AcNegCone_closed} ensures that to compute the actual negative normal cone $\Sigma(f_1, \dots , f_k)$, it is enough to check for finitely many $v \in \mathbb{R}^{n}$ whether
\begin{align}
\label{Eq_FiniteCheck}
\bigcap_{i=1}^{k} f_{i|\N(f_i)_v}^{-1}(\mathbb{R}_{<0} ) \neq \emptyset.
\end{align}
To be precise, for each cone $C \in \bigwedge_{i=1}^{k} \mathcal{N}_{f_i}$, pick $v \in \relint(C)$ and check whether \eqref{Eq_FiniteCheck} is empty. This observation provides a way to compute $\Sigma(f_1, \dots ,f_k)$. Note that, computing the negative normal cone $\mathcal{N}^-_{f_i}$ is significantly simpler and in some cases still gives enough information about the real tropicalization (Theorem \ref{Thm_Inclusions}, Corollary \ref{Cor_MaxSparse}). 

\medskip

We conclude the section with a bound on the roots of a univariate signomial that will be used in the proofs of Theorem \ref{Thm_Inclusions} and Theorem \ref{Lemma_SigmaTrop}. The statement is a generalization of Cauchy's bound \cite[Theorem 8.1.7]{rahman2002analytic} to signomials. Although the idea of the proof is almost the same as in the polynomial case, we recall it for the sake of completeness.

\begin{lemma}
\label{Lemma_RootLimitNew}
Let $g\colon \mathbb{R}_{>0} \to \mathbb{R},\, g(t) := \sum_{i = 1}^{d} a_i t^{\nu_i}$ be a univariate signomial such that $\nu_1 < \dots < \nu_d$.
If there exist  $p \in \{1, \dots , d\}$ and $\epsilon >0$ such that $(a_d - \epsilon) t^{\nu_d} + \sum_{i=p}^{d-1} a_i t^{\nu_i} > 0$ for all $t > 0$, then for all $0 < \delta \leq \nu_p - \nu_{p-1}$ it holds:
\[ \max \{ t_0 \in \mathbb{R}_{>0} \mid g(t_0) = 0 \} \leq \max \Big\{ 1, \Big( \tfrac{1}{\epsilon} \sum_{i=1}^{p-1} \lvert a_{i} \rvert \Big)^{\tfrac{1}{\delta}} \Big\}. \]
\end{lemma}

\begin{proof}
Let $t_0 > 1$ be such that $g(t_0) = 0$. Then,
\[ \epsilon \, t_0^{\nu_d} < \sum_{i = p}^{d} a_i t_0^{\nu_i} = - \sum_{i = 1}^{p-1} a_i t_0^{\nu_i} \leq \sum_{i = 1}^{p-1} \lvert a_i \lvert t_0^{\nu_i} \leq \sum_{i = 1}^{p-1} \lvert a_i \lvert t_0^{\nu_d - \delta} .\]
The first inequality follows from $(a_d - \epsilon) t_0^{\nu_d} + \sum_{i=p}^{d-1} a_i t_0^{\nu_i} > 0$. The last inequality follows from $\nu_d - \delta \geq \nu_{i}$ for all $i = 1, \dots , p-1$. Dividing both sides by $\epsilon \, t_0^{\nu_d - \delta} $, the statement follows.
\end{proof}

\begin{remark}
In Section \ref{Sec_RealTrop}, Lemma \ref{Lemma_RootLimitNew} will be used to ensure that for a signomial $f$ and a convergent sequence $w(m) \to w$ in $\mathbb{R}^{n}$, the roots of $f(t^{w(m)})$ converge to the roots of $f(t^{w})$. 

This statement might fail for some signomials not satisfying the assumptions of Lemma \ref{Lemma_RootLimitNew} as the following example shows. Consider the signomial $f(x_1,x_2) = x_1^2 - x_1  + 1 - x_2^2$ from Figure~\ref{FIG1} and the sequence $w(m) = (1 , 1-\tfrac{1}{m}) \to w = (1,1)$. The induced functions are given by:
\[f(t^w) = -t +1, \qquad f(t^{w(m)}) = t^2 - t^{2 -\tfrac{2}{m}} -t +1.\]
The function $f(t^w)$ has a unique root $t = 1$. The function $f(t^{w(m)})$ has two positive real roots, one of which converges to $1$ and the other goes to infinity as $m \to \infty$.
\end{remark}

\section{Real tropicalization}
\label{Sec_RealTrop}
The goal of this section is to relate the real tropicalization of the set
\begin{align}
\label{Eq_Sf}
S(f_1, \dots , f_k) := \bigcap_{i=1}^{k}  f_i^{-1}( \mathbb{R}_{<0}),
\end{align}
where $f_1, \dots ,f_k$ are signomials on $\mathbb{R}^{n}_{>0}$, to the negative normal cones $\mathcal{N}_{f_1}^{-}, \dots ,\mathcal{N}_{f_k}^{-}$, and to the actual negative normal cone $\Sigma(f_1, \dots ,f_k)$. In the case that \eqref{Eq_Sf} is described by a single polynomial $f$, we show that the real tropicalization and the actual negative normal cone coincide for generic coefficients of $f$. Furthermore, $\Trop(S(f))$ equals the negative normal cone of $f$ if $f$ is maximally sparse (all exponent vectors of  $f$ are vertices of the Newton polytope).

\medskip

As indicated in the Introduction, following \cite{RealTrop2022, BLEKHERMAN2022108561}, the \emph{real tropicalization} of $S \subseteq \mathbb{R}_{>0}^{n}$ is defined to be
\begin{align}
\label{Eq::DefOfLogLimit}
\Trop(S) := \lim_{t \to \infty} \Log_{t}(S) = \lim_{t \to \infty}\tfrac{1}{\log_e(t)} \Log_e(S),
\end{align}
where $\Log_t$ denotes the point-wise logarithm with base $t$, and $e$ is Euler's number. This construction is also called the \emph{logarithmic limit set} of $S$. For a precise definition and for some basic properties, we refer to \cite[Section 2]{Alessandrini2007LogarithmicLS}. Since it will be used frequently, we recall the following results from \cite{Alessandrini2007LogarithmicLS}:

\begin{prop}
\cite[Proposition 2.2]{Alessandrini2007LogarithmicLS} 
\label{Prop_Alessandrini}
For any set $S \subseteq \mathbb{R}^{n}_{>0}$, it holds:
\begin{itemize}
\item[(i)] $\Trop(S)$ is a closed subset of $\mathbb{R}^{n}$.
\item[(ii)] For $w \in \mathbb{R}^{n}$, $w \in \Trop(S)$ if and only if there exist sequences $\big\{x(m)\big\}_{m\in \mathbb{N}} \subseteq S$ and $\big\{t(m)\big\}_{m\in \mathbb{N}}  \subseteq (1,\infty)$ such that $t(m) \to \infty$ and $\Log_{t(m)}(x(m)) \to w$.
\item[(iii)] $\Trop(S) \subset \{0\}$ if and only if $\Log_t(S)$ is bounded for all $t >0$.
\end{itemize}
\end{prop}

As discussed in the Introduction, if $f_1, \dots ,f_k$ are polynomials, that is their exponents are nonnegative integer vectors, it is known that
\begin{align}
\label{Eq_InclPoly2}
\Trop\big(S(f_1, \dots , f_k)\big) =\bigcap_{i=1}^{k} \mathcal{N}_{f_i}^-
\end{align}
if the set on the right is a regular set (it is equal to the closure of its interior) \cite[Corollary 4.8]{TropSpecta}. The following proposition characterizes the cases when the intersection of the negative normal cones in \eqref{Eq_InclPoly2} is regular.

\begin{prop}
\label{Lemma_CommonNegVertex}
Let $f_1, \dots , f_k$ be signomials on $\mathbb{R}^{n}_{>0}$. A point $w \in  \bigcap_{i=0}^{k} \mathcal{N}_{f_i}^-$ is in the closure of $\inte\big(  \bigcap_{i=0}^{k} \mathcal{N}_{f_i}^- \big)$ if and only if there exists $v\in \mathbb{R}^{n}$ such that $\N(f_i)_{v}$ is a negative vertex of  $\N(f_i)_{w}$ for all $i = 1, \dots ,k$.
\end{prop}

\begin{proof}
First, we show the if part. For each $i = 1, \dots ,k$, let $\beta_i := \N(f_i)_{v}$ be the negative vertex of $\N(f_i)_{w}$. It follows that $w$ is contained in
\begin{align}
\label{Eq_ProofLemmaNegVertex}
\bigcap_{i=1}^{k} \mathcal{N}_{f_i}(\beta_i).
\end{align}
In the following, we show that the interior of \eqref{Eq_ProofLemmaNegVertex} is non-empty. As vertices correspond to full-dimensional normal cones by \eqref{Eq_DimCorr}, the interior and the relative interior of $\mathcal{N}_{f_i}(\beta_i)$ coincide. By \eqref{Eq_BdCorr}, $v \in \relint( \mathcal{N}_{f_i}(\beta_i)) = \inte(\mathcal{N}_{f_i}(\beta_i))$, which implies that there exists an $\epsilon >0$ such that the ball with radius $\epsilon$ and center $v$ is contained in the interior of $\mathcal{N}_{f_i}(\beta_i)$ for all $i = 1, \dots ,k$. Thus, the interior of  \eqref{Eq_ProofLemmaNegVertex} is non-empty.

As the finite intersection of interiors equals the interior of the intersection \cite[Section 1.1]{Topo_Book_CountE}, we have
\begin{align}
\label{Eq_ProofLemmaNegVertex2}
\inte\big( \bigcap_{i=1}^{k} \mathcal{N}_{f_i}(\beta_i) \big) = \bigcap_{i=1}^{k} \inte \big( \mathcal{N}_{f_i}(\beta_i) \big) \subseteq \bigcap_{i=1}^{k} \inte \big( \mathcal{N}^-_{f_i}\big) = \inte \big( \bigcap_{i=1}^{k} \mathcal{N}^-_{f_i} \big),
\end{align}
where the inclusion in the middle holds since $\beta_i$ is a negative vertex of $\N(f_i)$ for all $i=1, \dots ,k$.

Since \eqref{Eq_ProofLemmaNegVertex} is a polyhedral cone with non-empty interior, it is regular and $w$ is contained in the closure of the interior of \eqref{Eq_ProofLemmaNegVertex}. Using \eqref{Eq_ProofLemmaNegVertex2}, one concludes that $w$ is contained in the closure of $ \inte \big( \bigcap_{i=1}^{k} \mathcal{N}^-_{f_i} \big)$.

\medskip

To show the reverse implication, we assume that $w$ is in the closure of $\inte\big( \bigcap_{i=0}^{k} \mathcal{N}^-_{f_i} \big)$. In that case, for every $\epsilon >0$ we have that $B_{\epsilon}(w) \cap\inte\big( \bigcap_{i=0}^{k} \mathcal{N}^-_{f_i} \big) \neq \emptyset$, where $B_{\epsilon}(w)$ denotes the ball with radius $\epsilon$ and center $w$.  We choose $\epsilon$ small enough such that $B_{\epsilon}(w)$ does not intersect the cones of the $(n-1)$-skeleton of $\bigwedge_{i=1}^{k} \mathcal{N}_{f_i}$ that do not contain $w$.

The $(n-1)$-skeleton of $\bigwedge_{i=1}^{k} \mathcal{N}_{f_i}$ is not full-dimensional, so there exists $v \in B_{\epsilon}(w) \cap\inte\big( \bigcap_{i=0}^{k} \mathcal{N}^-_{f_i} \big)$ which is not contained in the $(n-1)$-skeleton.

Let $C$ be the smallest cone in $\bigwedge_{i=1}^{k} \mathcal{N}_{f_i}$ that contains $v$ in its interior, and let $C_i \in \mathcal{N}_{f_i}$, $i = 1 , \dots ,k$ such that $C = \cap_{i=1}^k C_i$. From $v \in \inte(C)$ follows that $v \in \inte(C_i)$ for all $i=1,\dots,k$. From \eqref{Eq_BdCorr} and \eqref{Eq_DimCorr} it follows that $\N(f_i)_v$ is a vertex for all $i=1,\dots,k$. Moreover, $\N(f_i)_v$ is a negative vertex, since $v \in \mathcal{N}_{f_i}^-$.

Since  $B_{\epsilon}(w)$ intersects only the cones of the $(n-1)$-skeleton of $\bigwedge_{i=1}^{k} \mathcal{N}_{f_i}$ that contain $w$, it follows that $w$ lies in $C$. Using \eqref{Eq_FaceConeCorr}, we conclude that $\N(f_i)_v \subseteq \N(f_i)_w$.
\end{proof}

For an example of a signomial whose negative normal cone is not regular, we refer to Figure~\ref{FIG1}. The negative normal cone of $f(x_1,x_2) = x_1^2 - x_1  + 1 - x_2^2$ has a ray in southern direction, but the closure of $\inte( \mathcal{N}_f^-)$ does not contain this ray. One can also use Proposition  \ref{Lemma_CommonNegVertex} to conclude that $\mathcal{N}_f^-$ is not regular: The face $\Conv\big( (0,0),(2,0) \big) \subseteq \N(f)$ is negative, but does not contain any negative vertex.

\medskip

In \cite{TropSpecta,RealTrop_SemiAlgSet}, to prove \eqref{Eq_InclPoly} the authors worked with polynomials over the field of real Puiseux series. 
In the following theorem, we give an elementary proof of \eqref{Eq_InclPoly} that applies to signomials and extend \eqref{Eq_InclPoly} by relating the real tropicalization to the actual negative normal cone as defined in \eqref{Eq_AcNegNormalCone}.

\begin{thm}
\label{Thm_Inclusions}
For signomials $f_1, \dots , f_k$ on $\mathbb{R}^{n}_{>0}$, it holds:
\[\inte\big(  \bigcap_{i=1}^{k} \mathcal{N}^-_{f_i} \big) \subseteq \Sigma(f_1, \dots , f_k) \subseteq \Trop\big( S(f_1, \dots , f_k) \big) \subseteq \bigcap_{i=1}^{k} \mathcal{N}^-_{f_i}. \]
\end{thm}

\begin{proof}
Let $S = S(f_1, \dots , f_k)$. To show the first inclusion, let $w \in \inte\big(  \bigcap_{i=1}^{k} \mathcal{N}^-_{f_i} \big)$. By Proposition \ref{Lemma_CommonNegVertex},  there exists $v\in \mathbb{R}^{n}$ such that $\N(f_i)_v$ is a negative vertex of $\N(f_i)_w$ for all $i = 1, \dots , k$. For any fixed $x \in \mathbb{R}^{n}_{>0}$, we have
\[f_{i|\N(f_i)_w}(t^v \ast x) < 0 \qquad  \text{for} \quad t \gg 0\]
by Lemma \ref{Lemma_Pushing}. Thus, $w \in  \Sigma(f_1, \dots ,f_k)$.

\medskip

For the proof of the second inclusion, let $w \in \Sigma(f_1, \dots ,f_k)$, and let $x \in  \bigcap_{i=1}^{k} f_{i|\N(f_i)_w}^{-1}(\mathbb{R}_{<0} )$. By Lemma \ref{Lemma_Pushing}, $f_i(t^w \ast x) < 0$ for all $i = 1 , \dots ,k$ and $t \gg 0$. Therefore, $t^w \ast x \in S$ for $t \gg 0$. Choose a sequence $(t(m))_{m \in \mathbb{N}}$ such that $t(m) \to \infty$  and $t(m)^w \ast x \in S$ for all $m \in \mathbb{N}$. Then
\[ \Log_{t(m)} ( t(m)^w \ast x ) = \log_{t(m)}(t(m))  w  +   \Log_{t(m)} ( x)   = w + \tfrac{1}{\log_e(t(m))}\Log_e(x) \to w .\]
Proposition \ref{Prop_Alessandrini} implies that $w \in \Trop(S)$.

\medskip

The third inclusion remains to be shown. For that, let $w \in \Trop(S)$. By Proposition \ref{Prop_Alessandrini}, there exist sequences $\{x(m) \}_{m \in \mathbb{N}} \subseteq S$ and $\{ t(m) \}_{m \in \mathbb{N}} \subseteq  \mathbb{R}_{>0}$ such that $t(m) \to \infty$ and $w(m) := \Log_{t(m)}(x(m)) \to w$. Note that with this notation we have
\[t(m)^{w(m)} = x(m), \qquad \text{for all } m \in \mathbb{N} .\]

 If $w \notin \bigcap_{i=1}^{k} \mathcal{N}^-_{f_i}$, then there exists $i \in \{1, \dots , k\}$ such that $w \in \mathbb{R}^{n} \setminus  \mathcal{N}^-_{f_{i}}$. In the following, we show that it is possible to choose a subsequence $\{ \tilde{w}(m) \}_{m \in \mathbb{N}}$ of $\{ w(m) \}_{m \in \mathbb{N}}$ such that
\[f_{i}(t^{\tilde{w}(m)}) > 0 \qquad \text{for all} \quad t > T,\]
where $T >0$ can be chosen independently from $m$. This yields to a contradiction, since for large $m$, $t(m) > T$ and $f_{i}(t(m)^{\tilde{w}(m)})< 0$.

Since the relative interiors of the cones in  $\mathcal{N}_{f_{i}}$ form a partition of $\mathbb{R}^{n}$ and there are only finitely many such cones, there exists a face $F \subseteq \N(f_{i})$ such that $\relint(\mathcal{N}_{f_{i}}(F))$ contains infinitely many elements of $\{w(m)\}_{m \in \mathbb{N}}$. Thus, we can pass to a subsequence $\{\tilde{w}(m)\}_{m \in \mathbb{N}}$ such that $F = \N(f_{i})_{\tilde{w}(m)}$ for all $m \in \mathbb{N}$. Since $\tilde{w}(m) \to w$ and $\mathcal{N}_{f_{i}}(F)$  is closed, $w \in \mathcal{N}_{f_{i}}(F)$. Therefore $F \subseteq \N(f_{i})_w$. Note that equality holds if and only if $w \in \relint\big(\mathcal{N}_{f_{i}}(F) \big)$  by \eqref{Eq_BdCorr}.

Since $w \in \mathbb{R}^{n} \setminus \mathcal{N}_{f_i}^-$ and this set is open in $\mathbb{R}^{n}$, it follows that $\tilde{w}(m) \in \mathbb{R}^{n} \setminus \mathcal{N}_{f_{i}}^-$ for all $m \gg 1$. Since $\N(f_i)_{\tilde{w}(m)} = F$ for all $m \in \mathbb{N}$, it follows that $\tilde{w}(m) \in \mathbb{R}^{n} \setminus \mathcal{N}_{f_{i}}^-$ for all $m \in \mathbb{N}$ and the face $F$ does not contain any negative exponent vector of $f_i$. This implies:
\begin{align}
\label{Eq_LemmaProofLC}
f_{i|N(f_{i})_{\tilde{w}(m)}}(1,\dots,1) = \sum_{\mu \in \sigma(f_{i}) \cap F} c_{\mu} = \sum_{\mu \in \sigma_+(f_{i}) \cap F} c_{\mu} > 0,
\end{align}
where $c_\mu$ for $\mu \in \sigma(f_{i})$ are the coefficients of $f_{i}$. Thus, by Lemma \ref{Lemma_Pushing} for each $\tilde{w}(m)$ there exists $T(m) > 0$ such that 
\[ f_{i}(t^{\tilde{w}(m)}) > 0, \quad \text{ for all } \quad t > T(m).\]

In the following, we show that $T(m)$ can be chosen independently from $m$. The leading coefficient of $f_{i}(t^{\tilde{w}(m)})$ is as given in \eqref{Eq_LemmaProofLC}, which is positive, so there exists $\epsilon > 0$ such that $f_{i|N(f_{i})_{\tilde{w}(m)}}(1,\dots,1) - \epsilon > 0$. Since $w \in \mathbb{R}^n \setminus \mathcal{N}_{f_{i}}^-$, all the exponent vectors on the face $\N(f_{i})_w$ are positive, i.e.
\[ c_\mu > 0, \qquad \text{for all } \mu \in \N(f_{i})_w \cap \sigma(f_{i}) .\]
Thus, the expression
\[ \big(f_{i|N(f_{i})_{\tilde{w}(m)}}(1,\dots,1) - \epsilon \big) t^{d} +\sum_{\mu \in (\N(f_{i})_w \setminus F) \cap \sigma(f_{i})}c_{\mu} t^{\mu \cdot \tilde{w}(m)},\]
where $d = \max_{\mu \in \N(f_{i})} \tilde{w}(m) \cdot \mu$, is positive for all $t>0$ and $m$. 

Note that there exists $\delta >0$ such that
\[w \cdot \mu - w \cdot \nu > \delta \qquad \text{for all } \quad \mu \in \sigma(f_{i}) \cap \N(f_{i})_w, \quad \nu \in \sigma(f_{i}) \setminus \N(f_{i})_w .\]
Since $\tilde{w}(m) \to w$ and the scalar product is continuous,
\[\tilde{w}(m) \cdot \mu - \tilde{w}(m) \cdot \nu > \delta \qquad \text{for all } \quad \mu \in \sigma(f_{i}) \cap \N(f_{i})_w, \quad \nu \in \sigma(f_{i}) \setminus \N(f_{i})_w \]
holds for $m$  large enough. By passing to a subsequence if necessary, we may assume that the above inequality holds for all $\tilde{w}(m)$. Lemma \ref{Lemma_RootLimitNew} gives a bound $T \geq 1$ on the positive roots of $f_{i}(t^{\tilde{w}(m)})$ that depends only on $\epsilon, \delta$ and $c_\mu$, $\mu \in \sigma(f_{i})$. In particular, this bound is independent of $m$. Since the leading coefficient of $f_i(t^{\tilde{w}(m)})$ is positive for all $m \in \mathbb{N}$ by \eqref{Eq_LemmaProofLC}, we conclude that
\[ f_i(t^{\tilde{w}(m)}) > 0.\]
for all $m \in \mathbb{N}$ and $t > T$.
\end{proof}

\begin{remark}
Some of the inclusions in Theorem \ref{Thm_Inclusions} can be generalized to semi-algebraic sets over a real closed field $\mathcal{R}$ with a compatible non-trivial non-Archimedean valuation $\val\colon \mathcal{R}^* \to \mathbb{R}$, for instance the field of real Puiseux series $\mathbb{R}\{\!\{t \}\!\}$.

Let $f_1, \dots , f_k \in \mathcal{R}[x_1, \dots , x_n]$ be polynomials and consider the semi-algebraic set
\[ S_\mathcal{R}(f_1, \dots , f_k) = \{ z \in \mathcal{R}_{>0}^n \mid f_1(z) < 0, \dots , f_k(z) <0 \}.\] 
Here, $>$ denotes the unique ordering on the real closed field $\mathcal{R}$.
The real tropicalization of a semi-algebraic set is defined as
\[ \Trop\big(  S_\mathcal{R}(f_1, \dots , f_k)  \big) := \overline{ \big\{ (-\val(z), \dots , -\val(z)) \mid z \in S_{\mathcal{R}} (f_1, \dots , f_k)\big\}},\]
where the closure is taken in the Euclidean topology of $\mathbb{R}^{n}$.

This construction generalizes the definition of  real tropicalization as a logarithmic limit in the following sense. If $\mathcal{R}$ is a non-Archimedean real closed field of rank one extending $\mathbb{R}$ (e.g. $\mathcal{R} = \mathbb{R}\{\!\{t \}\!\}$) and the coefficients of $f_1, \dots , f_k \in \mathcal{R}[x_1, \dots , x_n]$  are real numbers, then by \cite[Corollary 4.6]{Alessandrini2007LogarithmicLS} we have,
\[ \Trop\big( S(f_1, \dots , f_k)  \big) = \Trop\big(  S_\mathcal{R}(f_1, \dots , f_k)  \big),\]
where $S(f_1, \dots , f_k)$ is defined as in \eqref{Eq_Sf} and $\Trop(S(f_1, \dots , f_k))$ denotes the logarithmic limit of $S(f_1, \dots , f_k)$ as introduced in \eqref{Eq::DefOfLogLimit}.

By replacing the negative normal cones $\mathcal{N}_{f_i}^-$ with a  ``signed part'' of the tropical hypersurface  defined by $f_i$ (see \cite[Section 5.2]{RealTrop_SemiAlgSet} or \cite[Section 2.1]{RealTrop2022} for a precise definition), the inclusions in \eqref{Eq_InclPoly} remain true \cite[Proposition 6.12]{RealTrop_SemiAlgSet}.

To the best of our knowledge, there is no known generalization of the actual negative normal cone $\Sigma(f_1, \dots, f_k)$ in the non-trivial valuation case, such that the generalized objects would satisfy similar inclusions as in Theorem \ref{Thm_Inclusions}. The techniques used in the current paper do not allow immediately to find such a generalization. However, it is an interesting problem that might be addressed using alternative approaches.
\end{remark}

\begin{cor}
\label{Cor_MaxSparse}
For a maximally sparse signomial $f$, it holds:
\[ \Trop(S(f)) = \mathcal{N}_{f}^-.\]
\end{cor}

\begin{proof}
If all the exponent vectors of $f$ are vertices of $\N(f)$, then every negative face of $\N(f)$ must contain a negative vertex. Thus, $\mathcal{N}_f^-$ equals the closure of its interior by Proposition \ref{Lemma_CommonNegVertex}. Using that $\Trop(S(f))$ is closed and Theorem \ref{Thm_Inclusions}, we conclude that $\Trop(S(f)) = \mathcal{N}_f^-$.
\end{proof}

For a signomial $f$, it might be worth to know when $S(f)$ is bounded from the coordinate planes in $\mathbb{R}^{n}$ and from infinity. Note that this happens if and only if $\Log_t(S(f))$ is bounded for $t >0$, as the map $\Log_t \colon \mathbb{R}^{n}_{>0} \to \mathbb{R}$ is a homeomorphism.

\begin{cor}
\label{Cor_BoundedNegPre}
Let $f$ be a signomial on $\mathbb{R}^{n}_{>0}$. If $\sigma_-(f) \subseteq \inte(\N(f))$, then $\Log_t(S(f))$ is bounded for all $t >0$.
\end{cor}

\begin{proof}
If the boundary of $\N(f)$ does not contain negative monomials, then $\mathcal{N}_{f}^{-} \subseteq \{0\}$. Theorem \ref{Thm_Inclusions} implies that $\Trop(S(f)) \subseteq  \{0\}$.  Now, the statement follows from Proposition  \ref{Prop_Alessandrini}(iii).
\end{proof}

\begin{ex}
\label{Ex_boundedSet}
The boundary of the Newton polytope of $f = x_1^{9} x_2^{6} + x_1^{6} x_2^{9} - x_1^{7} x_2^{7} - 4 x_1^{7} x_2^{6} + 5  x_1^{5} x_2 + 5  x_1 x_2^{5} - 5  x_1 x_2 + 1$ does not contain any negative exponent vector of $f$. By Corollary \ref{Cor_BoundedNegPre}, $\Log_t\big( S(f) \big)$ is bounded for all $t >0$. For an illustration, we refer to Figure \ref{FIG2}.
\end{ex}

\begin{figure}[t]
\centering
\begin{minipage}[h]{0.3\textwidth}
\centering
\includegraphics[scale=0.4]{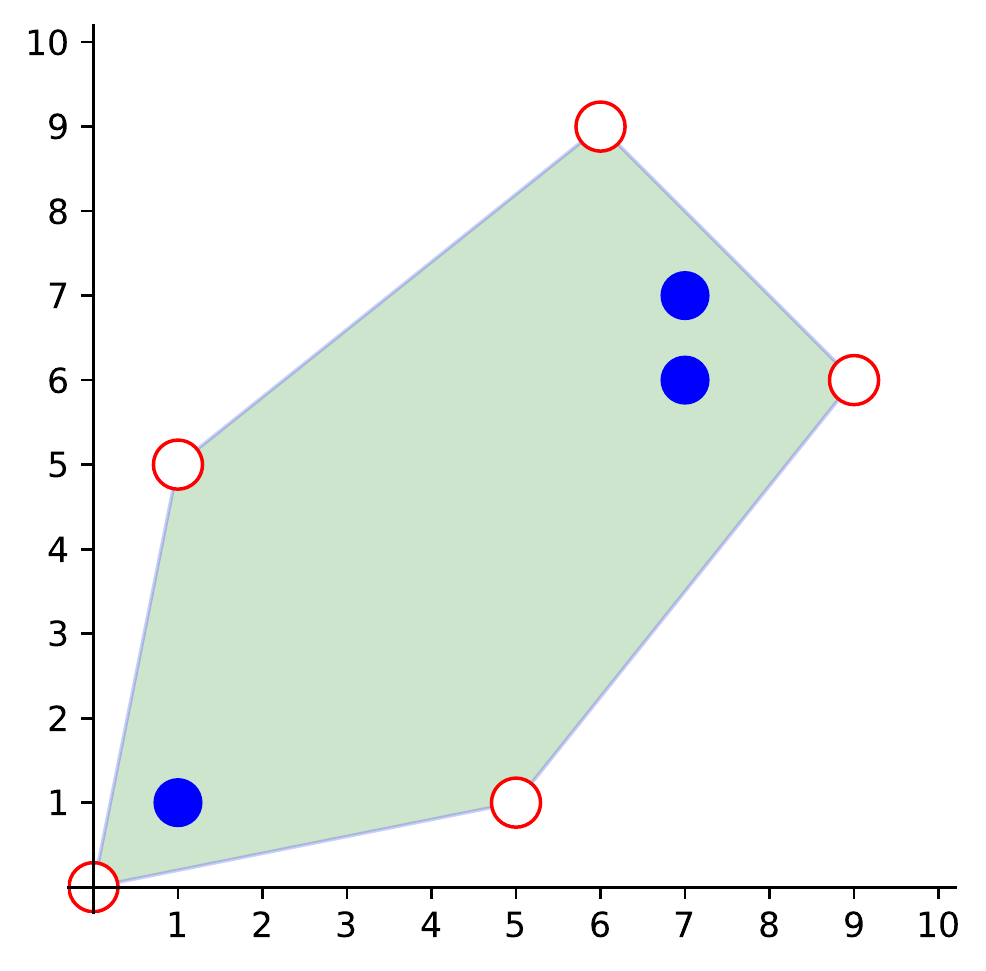}

{\small (a) $\N(f)$}
\end{minipage}
\begin{minipage}[h]{0.3\textwidth}
\centering
\includegraphics[scale=0.4]{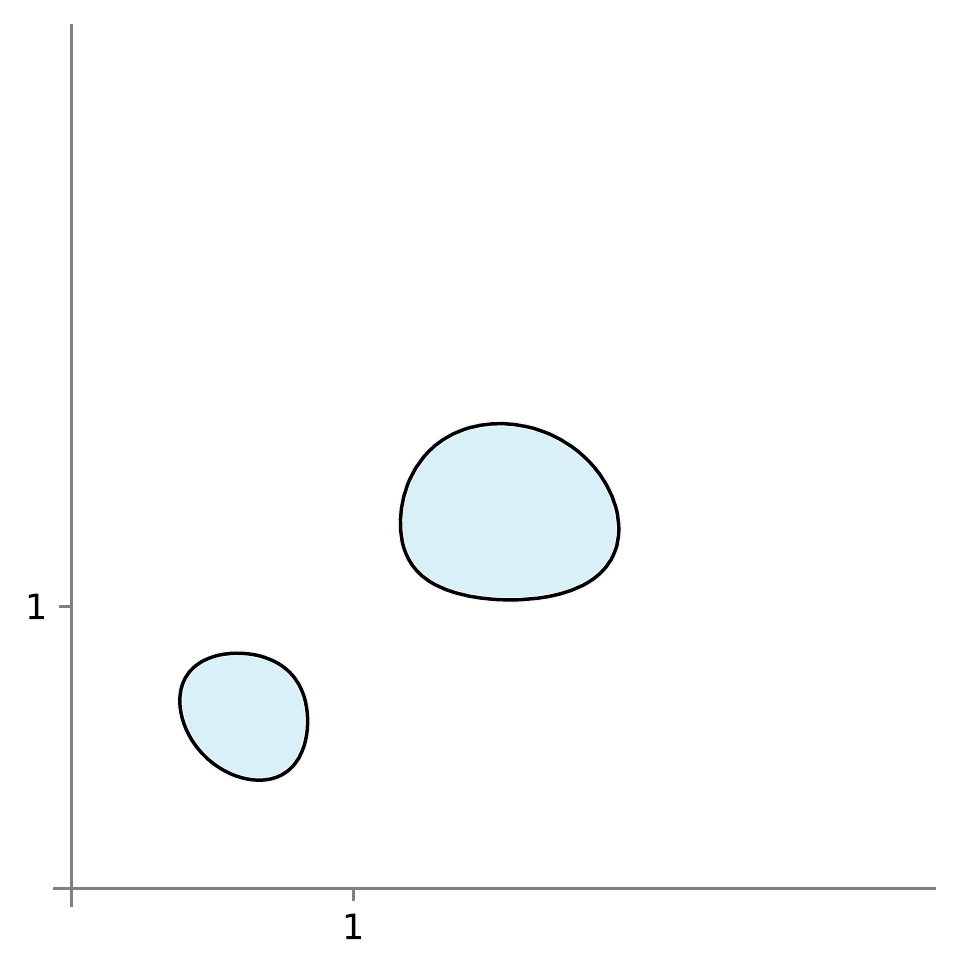}

{\small (b) $S(f)$}
\end{minipage}

\begin{minipage}[h]{0.3\textwidth}
\centering
\includegraphics[scale=0.4]{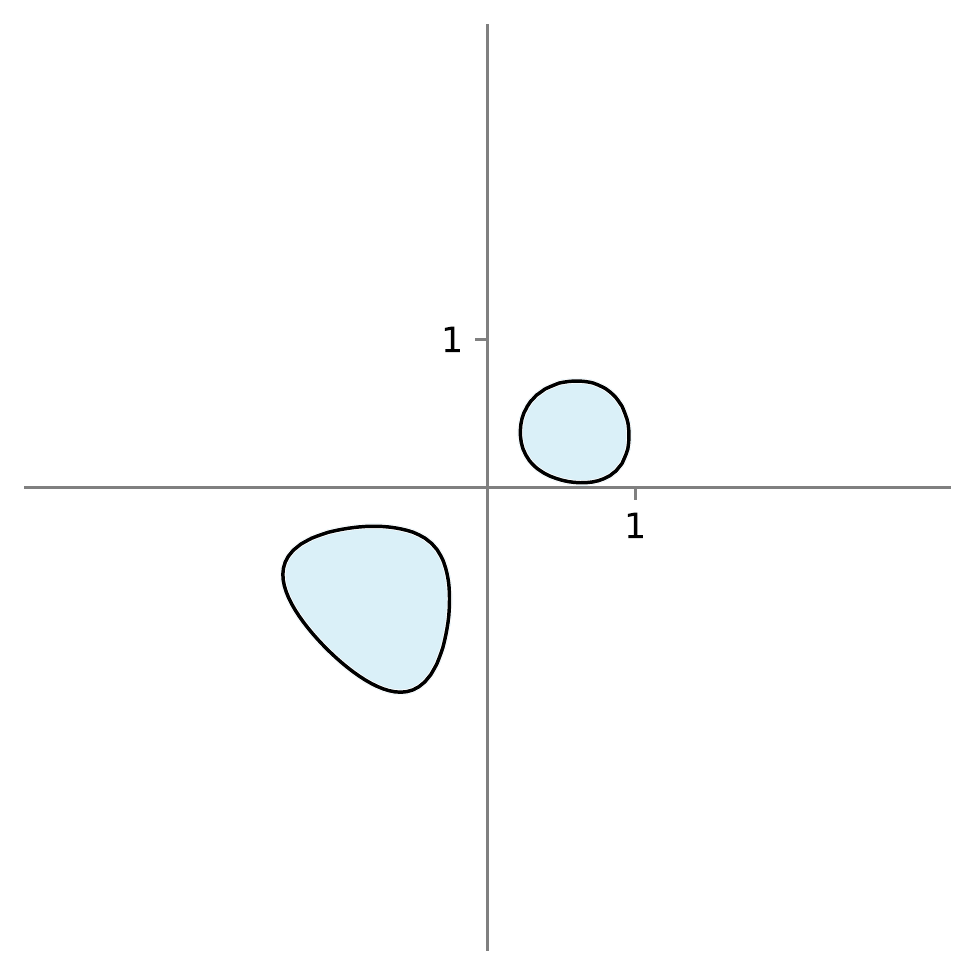}

{\small (c)  $\Log_2\big(S(f)\big)$}
\end{minipage}
\begin{minipage}[h]{0.3\textwidth}
\centering
\includegraphics[scale=0.4]{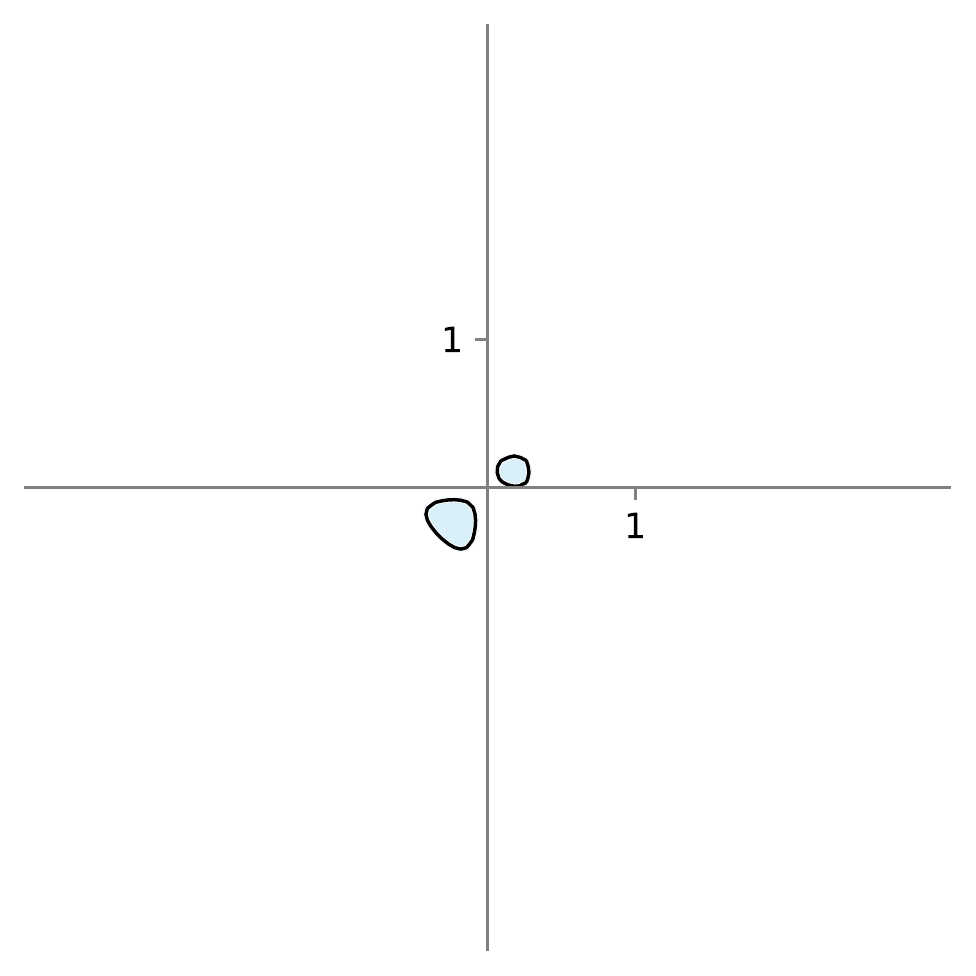}

{\small (d) $\Log_{10}\big(S(f)\big)$}
\end{minipage}
\caption{{\small (a) The Newton polytope of $f = x_1^{9} x_2^{6} + x_1^{6} x_2^{9} - x_1^{7} x_2^{7} - 4 x_1^{7} x_2^{6} + 5  x_1^{5} x_2 + 5  x_1 x_2^{5} - 5  x_1 x_2 + 1$ from Example \ref{Ex_boundedSet} (b) The set $S(f) = f^{-1}(\mathbb{R}_{<0})$ is bounded. (c)(d) Logarithmic images of $S(f)$ for $t=2$ and $t = 10$.}}\label{FIG2}
\end{figure}

\medskip

To conclude, we show that for a semi-algebraic set of the form $S(f)$ the second inclusion in Theorem \ref{Thm_Inclusions} is an equality if the coefficients of $f$ are generic. First, we clarify what we mean by generic coefficients. For a fixed finite set $A \subseteq \mathbb{R}^{n}$, we consider
the space of signomials whose support is contained in $A$:
\[ \mathbb{R}^{A} = \big\{ f = \sum_{\mu \in \sigma(f)} c_\mu x^{\mu} \mid \sigma(f) \subseteq A \big\}.\]
Interpreting the coefficients of signomials as vectors, one has an isomorphism $\mathbb{R}^{A} \cong \mathbb{R}^{\# A}$.

The set of nonnegative signomials in $\mathbb{R}^{A}$,
\[\mathcal{P}^{+}_{A} = \big\{ f \in \mathbb{R}^{A} \mid f(x) \geq 0 \quad \text{for all } x \in \mathbb{R}^{n}_{>0} \big\},\]
is a full-dimensional closed convex cone, called the \emph{nonnegativity cone} \cite{SoncBoundary, NonNegCone}. Using this terminology, the genericity condition we will assume is that for all faces $F \subseteq \N(f)$, the signomial $f_{|F}$ does not lie on the boundary of the nonnegativity cone, i.e.:
\begin{align}\tag{NB}
\label{Eq_NB}
 f_{|F} \in \mathcal{P}^{+}_{\sigma(f) \cap F}   \quad \Longrightarrow \quad f_{|F} \in \inte(\mathcal{P}^{+}_{\sigma(f) \cap F}  ).
\end{align}
As a consequence of the non-boundary assumption \eqref{Eq_NB}, we will be able to perturb the coefficients of a nonnegative signomial while preserving nonnegativity.

\begin{remark}
If a signomial, which is nonnegative over the positive real orthant, has a positive real root, then this root is degenerate and the signomial lies on the boundary of the nonnegativity cone. However, there exist signomials contained in the boundary of the nonnegativity cone that do not have any positive real roots as the following well-known example shows.

The signomial $g = x_1^2x_2^2 -2x_1x_2 + 1+x_1^2 = (x_1x_2-1)^2 + x_1^2$ is positive for all $(x_1,x_2) \in \mathbb{R}^{2}_{>0}$, but for any $\varepsilon >0$ the signomial $g_\varepsilon = x_1^2x_2^2 -(2+\varepsilon)x_1x_2 + 1+x_1^2$ takes negative values in $\mathbb{R}^2_{>0}$, e.g. $g_\varepsilon(\tfrac{\sqrt{\varepsilon}}{2},\tfrac{2}{\sqrt{\varepsilon}}) = \tfrac{3}{4} \varepsilon <0$. Thus, even if a nonnegative signomial does not have any roots, i.e. it is strictly positive, we might not be able to perturb its coefficients while preserving nonnegativity.
\end{remark}

\begin{thm}
\label{Lemma_SigmaTrop}
Let $f_1, \dots , f_k$ be signomials on $\mathbb{R}^{n}_{>0}$ such that condition \eqref{Eq_NB} holds for all $f_i$ and all faces $F \subseteq \N(f_i)$, $\, i=1, \dots ,k$. Then:
\[\Trop\big(S(f_1, \dots ,f_k)\big) \subseteq \bigcap_{i=1}^{k}\Sigma(f_i) .\]
\end{thm}

\begin{proof}
We use a similar argument to the proof of Theorem \ref{Thm_Inclusions}. Let $w \in \Trop(S)$. By Proposition \ref{Prop_Alessandrini}, there exist sequences $\{x(m) \}_{m \in \mathbb{N}} \subseteq S$, and $\{ t(m) \}_{m \in \mathbb{N}} \subseteq \mathbb{R}_{>0}$ such that $t(m) \to \infty$ and $w(m) := \Log_{t(m)}(x(m)) \to w$.

Assume that $ w\notin\bigcap_{j=1}^{k}\Sigma(f_j)$, and let $i \in \{1, \dots , k\}$ such that $w \notin \Sigma(f_{i})$. In the following, we write $c_\mu$, $\mu \in \sigma(f_{i})$ for the coefficients of $f_{i}$ and $G := \N(f_{i})_w$ for the face of $\N(f_{i})$ which is cut out by $w$. 
By an argument as in the proof of Theorem \ref{Thm_Inclusions}, we may pass to a subsequence if necessary and assume that there exists a face $F \subseteq G$ of $\N(f_{i})$ such that $F = \N(f_{i})_{w(m)}$ for all $m \in \mathbb{N}$. 

Since $w$ cuts out the face $G$, there exists $\delta > 0$ such that
\begin{align}
\label{Eq_SigmaDelta}
w \cdot \mu - \delta > w \cdot \nu, \qquad \text{for all } \mu \in G \cap \sigma(f_{i}), \quad \nu \in \sigma(f_{i}) \setminus G.
\end{align}
By continuity of the scalar product, \eqref{Eq_SigmaDelta} holds also for $w(m)$ for large enough $m$. Thus, we can pass to a subsequence again to assume that \eqref{Eq_SigmaDelta} holds for all $w(m)$.

Since $w \notin \Sigma(f_{i})$, the signomial $f_{i|G}$ is in the nonnegativity cone $\mathcal{P}^+_{\sigma(f_i) \cap F}$. By the condition \eqref{Eq_NB} $f_{i|G}$ is in the interior and hence there exists $\tilde{\epsilon} >0$ such that
\begin{align}
\label{Eq_FPositive}
 \sum_{\mu \in F \cap \sigma(f_{i})} (c_\mu - \tilde{\epsilon}) x^{\mu} + \sum_{\mu \in (G\setminus F) \cap \sigma(f_{i})} c_\mu  x^{\mu} > 0 \quad  \text{ for all } \quad x \in \mathbb{R}^{n}_{>0}.
\end{align}
In particular, it holds:
\begin{align}
\label{Eq_FPositive}
 \sum_{\mu \in F \cap \sigma(f_{i})} (c_\mu - \tilde{\epsilon}) t^{w(m) \cdot \mu} + \sum_{\mu \in (G\setminus F) \cap \sigma(f_{i})} c_\mu  t^{w(m) \cdot \mu} > 0 \quad  \text{ for all } t \in \mathbb{R}_{>0} \text{ and } m \in \mathbb{N}.
\end{align}
From \eqref{Eq_SigmaDelta} and \eqref{Eq_FPositive} follows that for each $m \in \mathbb{N}$ the univariate signomial 
\[f(t^{w(m)}) =  \Big( \sum_{\mu \in F \cap \sigma(f_{i})} c_\mu  \Big) t^{d} + \sum_{\mu \in (G\setminus F) \cap \sigma(f_{i})} c_\mu  t^{w(m) \cdot \mu} + \sum_{\mu \in \sigma(f_{i}) \setminus G} c_\mu  t^{w(m) \cdot \mu}, \]
where $d = \max_{\mu \in \sigma(f_{i})} w(m) \cdot \mu$, satisfies the conditions in Lemma \ref{Lemma_RootLimitNew}  with  $\epsilon = \sum_{\mu \in F \cap \sigma(f_{i})} \tilde{\epsilon} >0$ and $\delta$ as in  \eqref{Eq_SigmaDelta}. Thus, there exists $T > 1$ which is independent of $m$ such that
\[f_{i}(t^{w(m)}) > 0, \qquad \text{ for all } t > T.\]

Since $t(m) \to \infty$, $t(m) > T$ for $m \gg 1$. Thus,
\[ 0 <  f_{i}(t(m)^{w(m)}) = f_{i}(x(m)) <0.\]
which is a contradiction. The last inequality holds as $x(m) \in S(f_1, \dots , f_k)$.
\end{proof}

\begin{cor}
\label{Cor_RealTropOnePoly}
Let $f$ be a signomial on $\mathbb{R}^{n}_{>0}$. If all faces $F \subseteq \N(f)$ satisfy \eqref{Eq_NB} then:
\[\Trop\big( S(f) \big) = \Sigma(f).\]
\end{cor}

\begin{proof}
By Theorem \ref{Lemma_SigmaTrop}, $\Trop\big(S(f) \big) \subseteq \Sigma(f)$. The reverse inclusion follows from Theorem \ref{Thm_Inclusions}.
\end{proof}

\begin{figure}[t]
\centering
\begin{minipage}[h]{0.3\textwidth}
\centering
\includegraphics[scale=0.4]{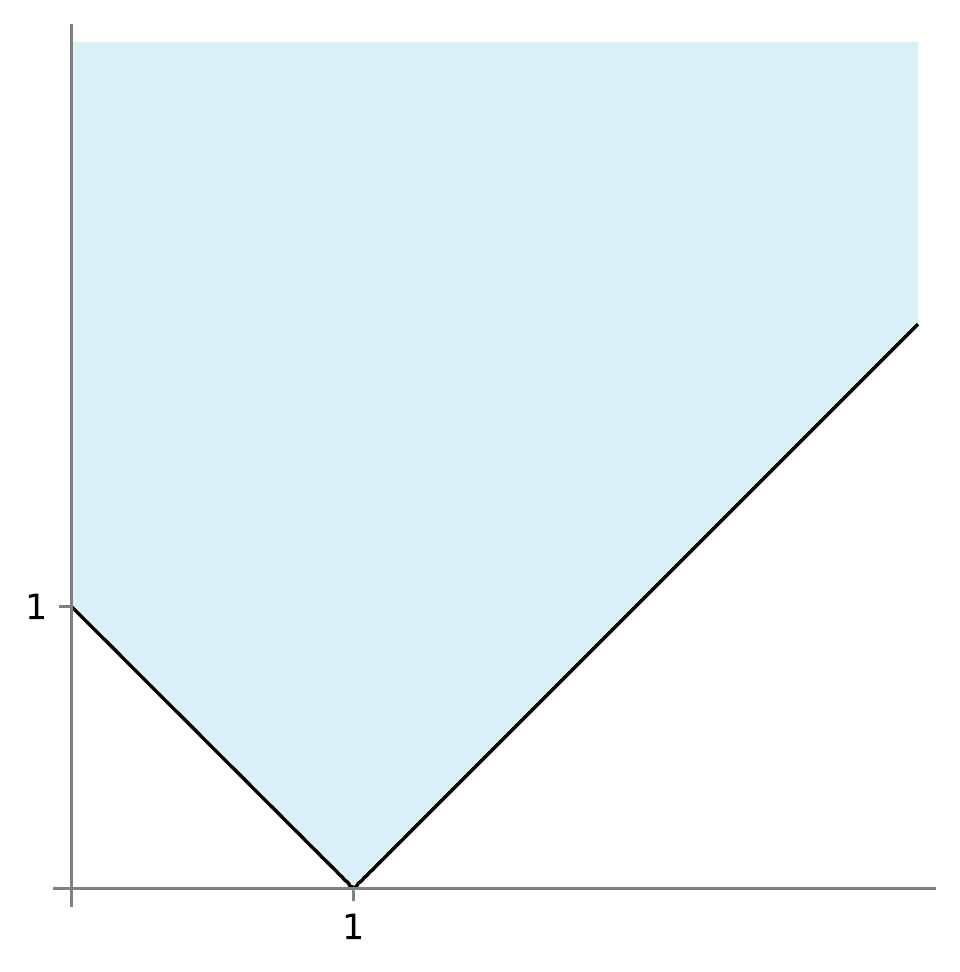}

{\small (a) $S(g)$ }
\end{minipage}
\begin{minipage}[h]{0.3\textwidth}
\centering
\includegraphics[scale=0.4]{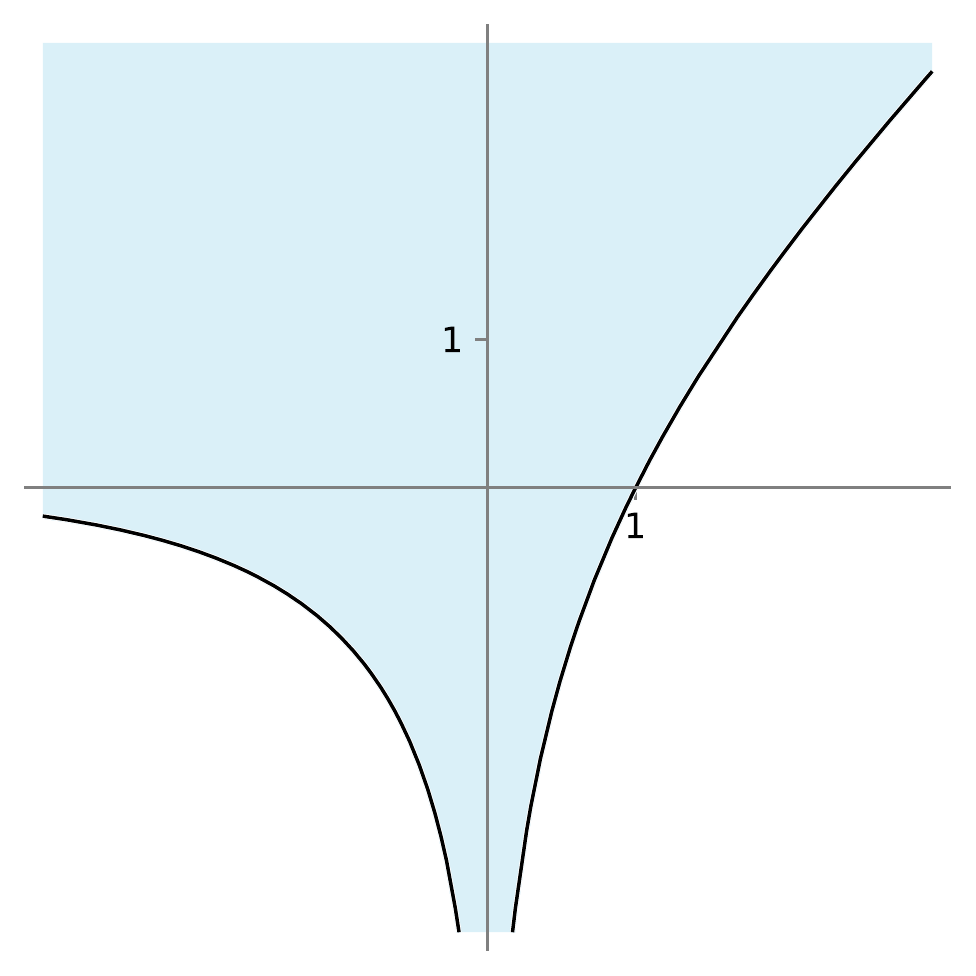}

{\small (b) $\Log_2\big( S(g) \big)$ }
\end{minipage}
\begin{minipage}[h]{0.3\textwidth}
\centering
\includegraphics[scale=0.4]{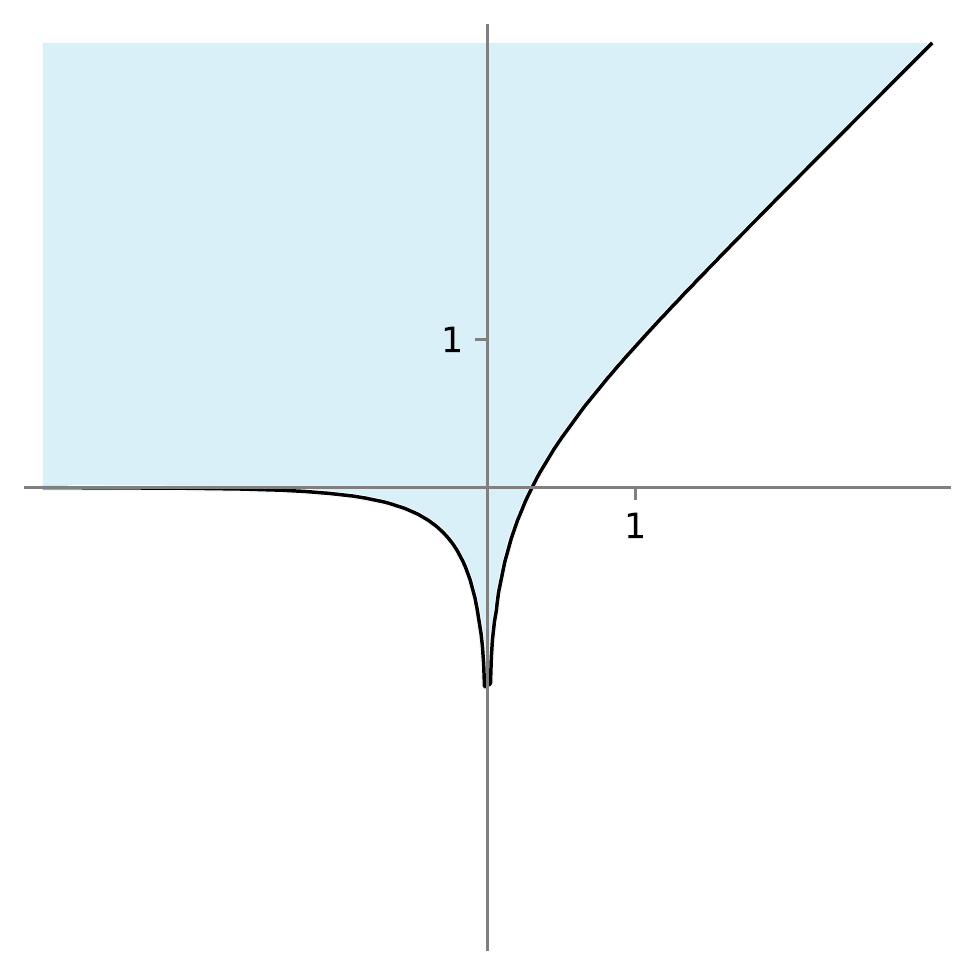}

{\small (c)  $\Log_{10}\big( S(g) \big)$}
\end{minipage}
\caption{{\small The semi-algebraic set $S(g)$ defined by $g = x_1^2-2x_1+1 - x_2^2$ (a), and its logarithmic images for $t=2$ (b) and $t=10$ (c).}}\label{FIG3}
\end{figure}

\begin{ex}
Consider the signomials $f =  x_1^2- x_1+1 - x_2^2$ from Figure~\ref{FIG1} and $g = x_1^2-2 x_1+1 - x_2^2$. Note that $f$ and $g$ have the same negative/positive exponent vectors and $\N(f) = \N(g)$. The only difference between $f$ and $g$ is the coefficient of $x_1$. Furthermore, we have:
\[ \Sigma(f) = \Sigma(g), \qquad \mathcal{N}^-_f = \mathcal{N}^-_g.\] 

For the faces $F_1 = \Conv ( (2,0) ,(0,2) )$,  $F_2 = \Conv ( (0,0) ,(0,2) )$ of $\N(f)$, the restrictions $f_{|F_i}$ and $g_{|F_i}$, $i=1,2$ take negative values by Lemma \ref{Lemma_Pushing}. Thus, $f_{|F_i}$ and $g_{|F_i}$, $i=1,2$ are not in the nonnegativity cone.

For the face $F = \Conv ( (0,0) ,(2,0) )$, $f_{|F} = x_1^2 - x_1 + 1$ is contained in the interior of the nonnegativity cone. So the coefficients of $f$ are generic in the sense of Corollary \ref{Cor_RealTropOnePoly}, thus:
\[ \Trop(S(f)) = \Sigma(f).\]
The sets $S(f), \Sigma(f)$ and $\Log_t(S(f))$ for $t=2,10$ are shown in Figure \ref{FIG1}. The negative normal cone $\mathcal{N}_f^-$  (see Figure\ref{FIG1}(c)) is strictly larger than $\Trop(S(f))$ as $\mathcal{N}_f^-$ has a ray in southern direction that $\Trop(S(f))$ does not have. This example illustrates that the negative normal cone might not coincide with the real tropicalization.

As $g_{|F}=  x_1^2 - 2x_1 + 1$ is nonnegative but has a positive real root, $g_{|F}$ lies on the boundary of the nonnegativity cone. Thus \eqref{Eq_NB} is not satisfied, and we cannot apply  Corollary \ref{Cor_RealTropOnePoly}. In fact, it holds that
\[ \Trop(S(g)) = \mathcal{N}_g^-,\]
 which is strictly larger than the actual negative normal cone $\Sigma(g)$, see Figure~\ref{FIG1} (b),(c). This illustrates that if the condition \eqref{Eq_NB} is not satisfied, Corollary \ref{Cor_RealTropOnePoly} might not hold. Logarithmic images of $S(g)$ with base $t=2,10$ are shown in Figure~\ref{FIG3}.
\end{ex}

\begin{ex}
\label{Ex::3d}
Consider the signomial 
\[f = x_2^2 - 2x_2 + 1 - 2 x_1x_2x_3 + x_1x_2x_3^2+x_1^2x_2.\] 
The Newton polytope of $f$ is shown in Figure \ref{FIG4}(a). The only negative exponent vector of $f$ that lies on the boundary of $\N(f)$ is $(0,1,0)$. This negative exponent vector is contained in the $1$-dimensional face $F = \Conv\big( (0,0,0) , (0,2,0) \big)$. The negative normal cone of $f$ is $2$-dimensional, and it is spanned by the vectors $(0,0,-1)$ and $(-1,0,2)$, see Figure\ref{FIG4}(b). Since $\mathcal{N}_f^-$ is not full dimensional, we have that $\inte(\mathcal{N}_f^-) = \emptyset$.

The restricted signomial $f_{|F} = x_2^2 - 2x_2 + 1$ is non-negative but it has a positive real zero. Thus, $f_{|F}$ does not satisfy the condition \eqref{Eq_NB}. Since $f(0.5,1,0.5) = -0.125 <0$ and $f_{|G}$ is non-negative for all proper faces $G \subsetneq \N(f)$, we have $\Sigma(f) = \{ 0 \}$. 

The real tropicalization of $S(f)$ is a  $2$-dimensional cone, spanned by the vectors $(-1,0,0)$ and $(-1,0,-1)$. This example shows that $\Trop(S(f))$ is not always a subfan of the outer normal fan of $\N(f)$, and that all the inclusions in Theorem \ref{Thm_Inclusions} might be strict:
\[\inte\big(   \mathcal{N}^-_{f} \big) \subsetneq \Sigma(f) \subsetneq \Trop\big( S(f) \big) \subsetneq  \mathcal{N}^-_{f}. \]
\end{ex}

\begin{figure}[t]
\centering
\begin{minipage}[h]{0.45\textwidth}
\centering
\includegraphics[scale=0.45]{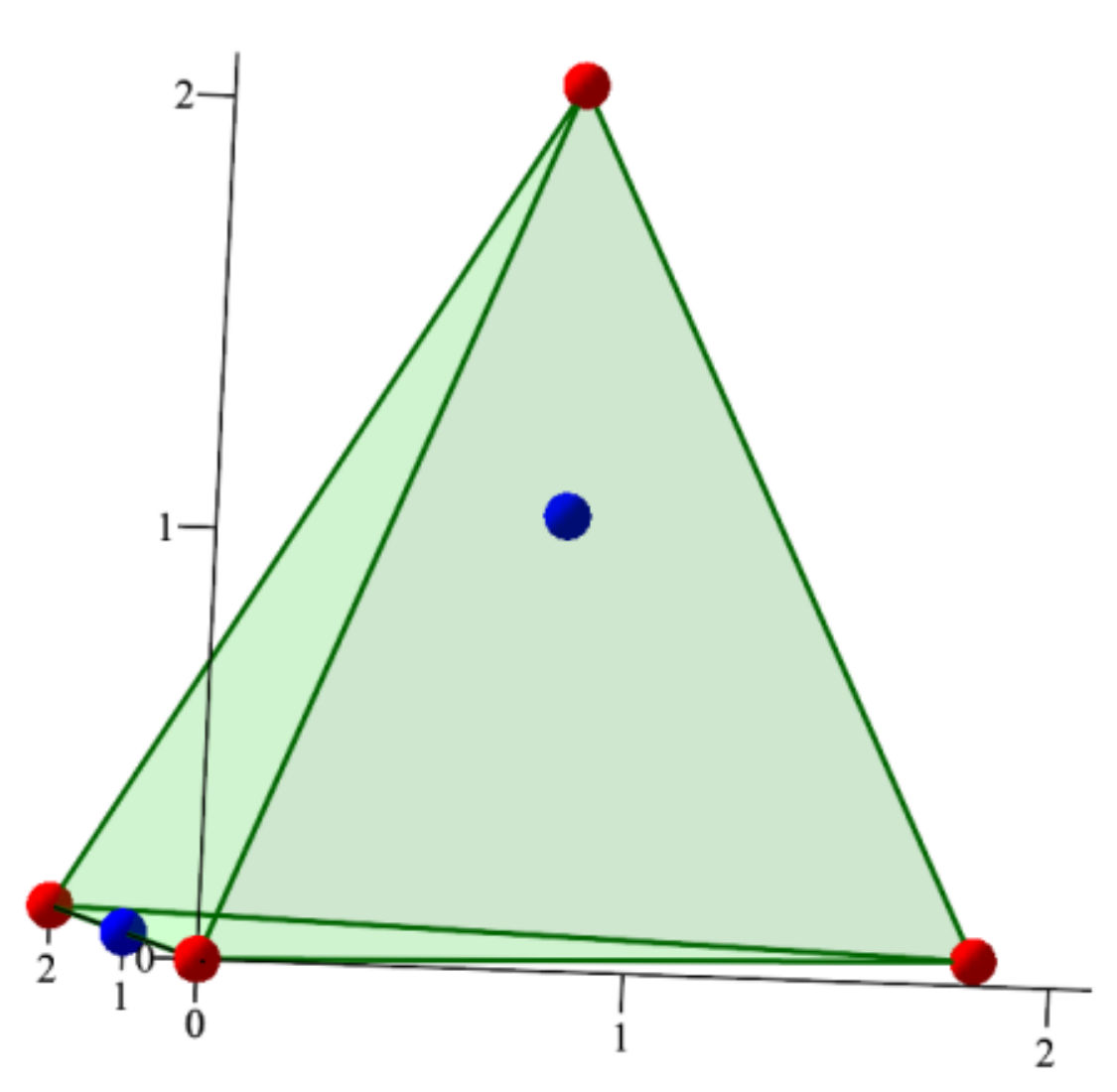}

{\small (a) $\N(f)$ }
\end{minipage}
\begin{minipage}[h]{0.45\textwidth}
\centering
\includegraphics[scale=0.45]{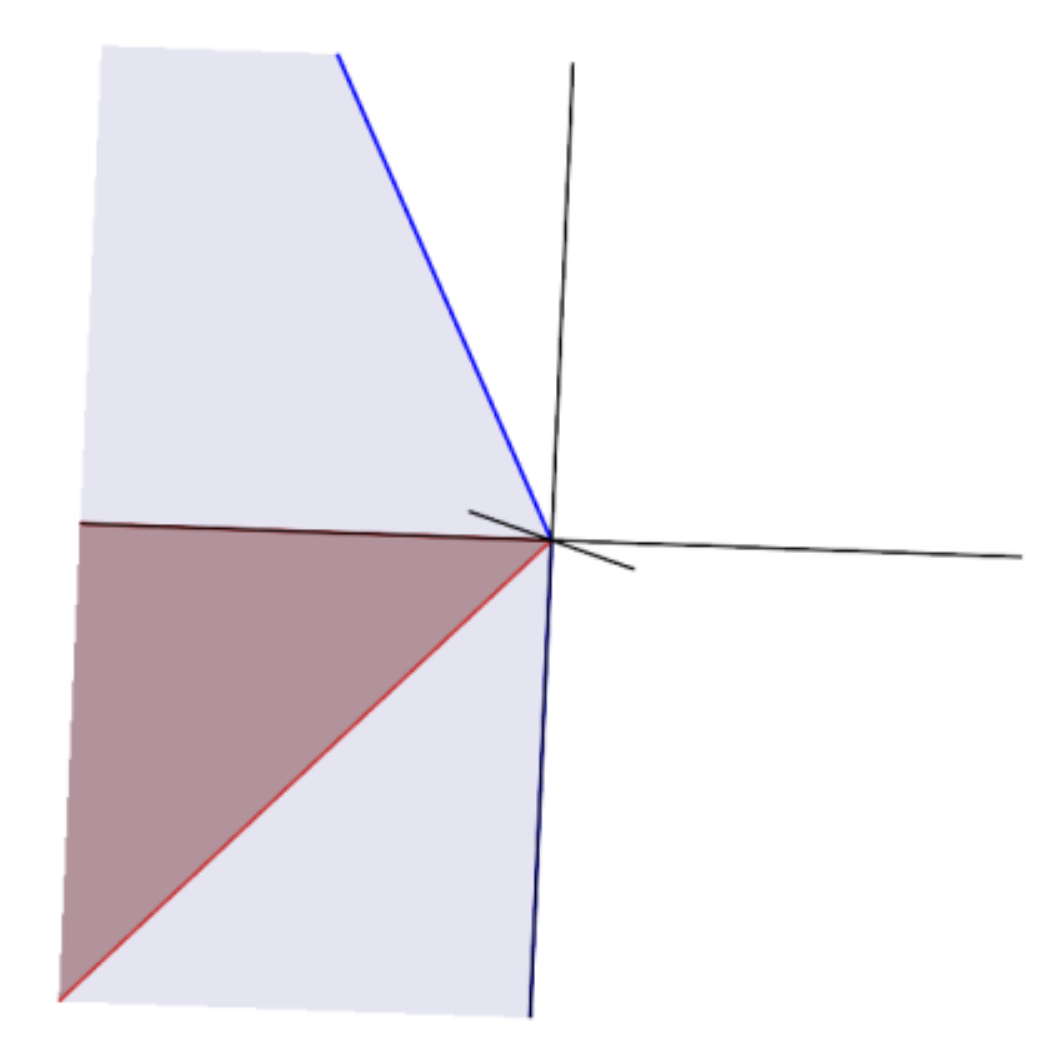}

{\small (b) $\Trop\big( S(f) \big) \subsetneq \mathcal{N}_f^-$}
\end{minipage}

\begin{minipage}[h]{0.45\textwidth}
\centering
\includegraphics[scale=0.35]{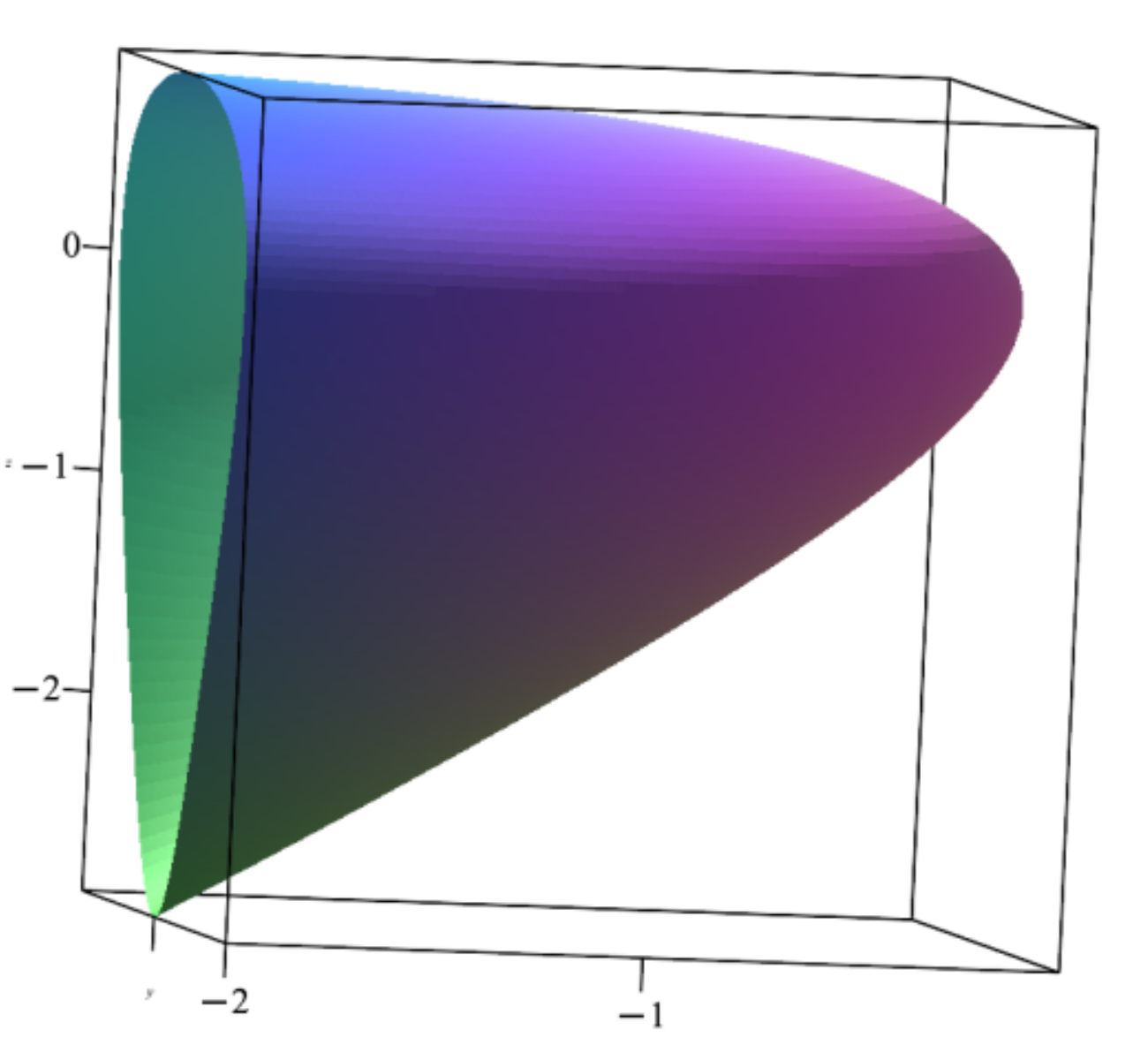}

{\small (c)  $\Log_{2}\big( S(f) \big)$}
\end{minipage}
\begin{minipage}[h]{0.45\textwidth}
\centering
\includegraphics[scale=0.35]{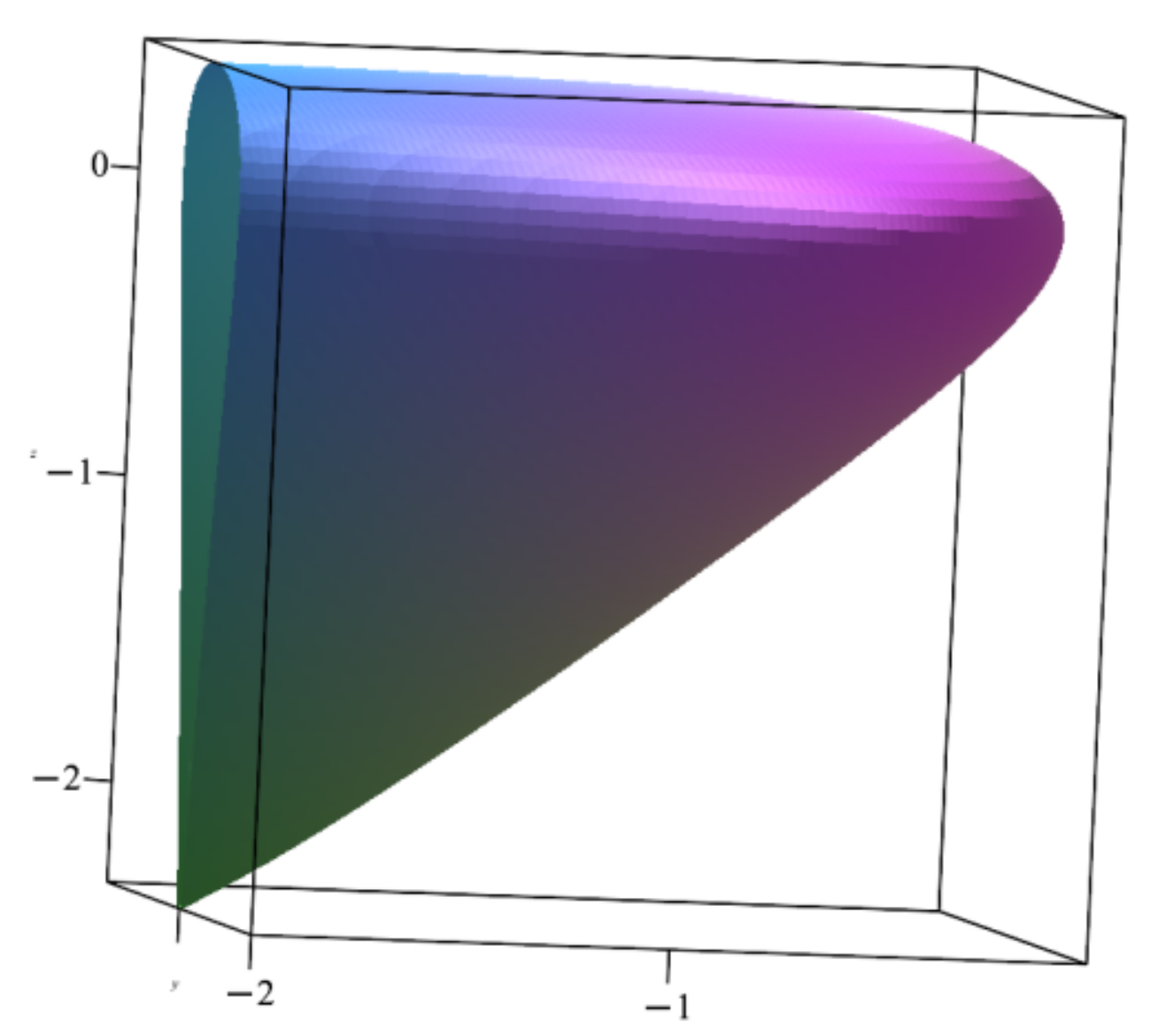}

{\small (d)  $\Log_{5}\big( S(f) \big)$}
\end{minipage}
\caption{{\small Illustration of Example \ref{Ex::3d} (a) The Newton polytope of  $f = x_2^2 - 2x_2 + 1 - 2 x_1x_2x_3 + x_1x_2x_3^2+x_1^2x_2$ . The exponent vector $(1,1,1)$ lies in the interior of $\N(f)$. (b) Negative normal cone of $f$ (blue shaded area) and the real tropicalization of $S(f)$ (pink shaded area) (c),(d) Logarithmic images of $S(f)$ for $t=2$ and $t=5$ respectively.}}\label{FIG4}
\end{figure}

\section*{Acknowledgments}
The author thanks Elisenda Feliu for useful discussions and comments on the manuscript. Funded by the European Union under the Grant Agreement no. 101044561, POSALG. Views and opinions expressed are those of the author(s) only and do not necessarily reflect those of the European Union or European Research Council (ERC). Neither the European Union nor ERC can be held responsible for them

\section*{Declaration of competing interest}
The authors declare that they have no known competing financial interests or personal relationships that could have appeared to influence the work reported in this paper.

{\small

\begin{thebibliography}{10}

\bibitem{Alessandrini2007LogarithmicLS}
D.~Alessandrini.
\newblock Logarithmic limit sets of real semi-algebraic sets.
\newblock {\em Adv. Geom.}, 13(1):155--190, 2013.

\bibitem{TropSpecta}
X.~Allamigeon, S.~Gaubert, and M.~Skomra.
\newblock Tropical spectrahedra.
\newblock {\em Discrete. Comput. Geom.}, 63(3):507--548, 2020.

\bibitem{Bergman_LogLimit}
G.~M. Bergman.
\newblock The logarithmic limit-set of an algebraic variety.
\newblock {\em Trans. Am. Math. Soc.}, 157:459--469, 1971.

\bibitem{BLEKHERMAN2022108561}
G.~Blekherman and A.~Raymond.
\newblock A path forward: Tropicalization in extremal combinatorics.
\newblock {\em Adv. Math.}, 407:108561, 2022.

\bibitem{RealTrop2022}
G.~Blekherman, F.~Rincón, R.~Sinn, C.~Vinzant, and J.~Yu.
\newblock Moments, sums of squares, and tropicalization.
\newblock {\em arXiv}, (2203.06291), 2022.

\bibitem{Chandrasekaran2014RelativeER}
V.~Chandrasekaran and M.~P. Shah.
\newblock Relative entropy relaxations for signomial optimization.
\newblock {\em SIAM J. Optim.}, 26:1147--1173, 2014.

\bibitem{CRN_Dickenstein}
A.~Dickenstein.
\newblock Algebraic geometry tools in systems biology.
\newblock {\em Not. Am. Math. Soc.}, 67:1, 12 2020.

\bibitem{DRESSLER2019149}
M.~Dressler, S.~Iliman, and T.~{de Wolff}.
\newblock An approach to constrained polynomial optimization via nonnegative
  circuit polynomials and geometric programming.
\newblock {\em Journal of Symbolic Computation}, 91:149--172, 2019.
\newblock MEGA 2017, Effective Methods in Algebraic Geometry, Nice (France),
  June 12-16, 2017.

\bibitem{duffin1973geometric}
R.~J. Duffin and E.~L. Peterson.
\newblock Geometric programming with signomials.
\newblock {\em J. Optimiz. Theory. App.}, 11(1):3--35, 1973.

\bibitem{MultDualPhos}
E.~Feliu, N.~Kaihnsa, T.~{de Wolff}, and O.~Y\"ur\"uck.
\newblock The kinetic space of multistationarity in dual phosphorylation.
\newblock {\em J. Dyn. Differ. Equ.}, 34:825--852, 2022.

\bibitem{SoncBoundary}
J.~Forsgård and T.~{de Wolff}.
\newblock The algebraic boundary of the sonc cone.
\newblock {\em arXiv}, (1905.04776), 2019.

\bibitem{YuePaul}
P.~A. Helminck and Y.~Ren.
\newblock Generic root counts and flatness in tropical geometry.
\newblock {\em arXiv}, (2206.07838), 2022.

\bibitem{BasicsOnPolytopes}
M.~Henk, J.~Richter-Gebert, and G.~M. Ziegler.
\newblock Basic properties of convex polytopes.
\newblock In {\em Handbook of Discrete and Computational Geometry, 3rd Ed.},
  2017.

\bibitem{NonNegCone}
J.~Heuer and T.~de~Wolff.
\newblock The duality of sonc: Advances in circuit-based certificates.
\newblock {\em arXiv}, (2204.03918), 2022.

\bibitem{Hilbert1888berDD}
D.~R. Hilbert.
\newblock {\"U}ber die darstellung definiter formen als summe von
  formenquadraten.
\newblock {\em Mathematische Annalen}, 32:342--350, 1888.

\bibitem{SONC}
S.~Iliman and T.~{de Wolff}.
\newblock Amoebas, nonnegative polynomials and sums of squares supported on
  circuits.
\newblock {\em Research in the Mathematical Sciences}, 3, 03 2016.

\bibitem{RealTrop_SemiAlgSet}
P.~Jell, C.~Scheiderer, and J.~Yu.
\newblock {Real Tropicalization and Analytification of Semialgebraic Sets}.
\newblock {\em Int. Math. Res. Notices}, 2022(2):928--958, 05 2020.

\bibitem{JoswigTheobald_book}
M.~Joswig and T.~Theobald.
\newblock {\em Polyhedral and Algebraic Methods in Computational Geometry}.
\newblock Springer, 01 2013.

\bibitem{Laurent2009}
M.~Laurent.
\newblock {\em Sums of Squares, Moment Matrices and Optimization Over
  Polynomials}, pages 157--270.
\newblock Springer New York, New York, NY, 2009.

\bibitem{maclagan2015introduction}
D.~Maclagan and B.~Sturmfels.
\newblock {\em Introduction to Tropical Geometry}.
\newblock Graduate Studies in Mathematics. American Mathematical Society, 2015.

\bibitem{Mikhalkin2003EnumerativeTA}
G.~Mikhalkin.
\newblock Enumerative tropical algebraic geometry in $\mathbb{R}^2$.
\newblock {\em J. Am. Math. Soc.}, 18:313--378, 2003.

\bibitem{MIKHALKIN20041035}
G.~Mikhalkin.
\newblock Decomposition into pairs-of-pants for complex algebraic
  hypersurfaces.
\newblock {\em Topology}, 43(5):1035--1065, 2004.

\bibitem{Nisse2008maximally}
M.~Nisse.
\newblock Maximally sparse polynomials have solid amoebas.
\newblock {\em arXiv}, (0704.2216), 2008.

\bibitem{rahman2002analytic}
Q.I. Rahman and G.~Schmeisser.
\newblock {\em Analytic Theory of Polynomials}.
\newblock London Mathematical Society monographs. Clarendon Press, 2002.

\bibitem{Reznick1989FormsDF}
B.~Reznick.
\newblock Forms derived from the arithmetic-geometric inequality.
\newblock {\em Mathematische Annalen}, 283:431--464, 1989.

\bibitem{Rockafellar}
R.~T. Rockafellar.
\newblock {\em Convex analysis}.
\newblock Princeton University Press, 1972.

\bibitem{FuzzyGeo}
I.~Sahidul and A.~M. Wasim.
\newblock {\em Fuzzy Geometric Programming Techniques and Applications}.
\newblock Springer, 2019.

\bibitem{Topo_Book_CountE}
L.~A. Steen and Seebach~J. A.
\newblock {\em Counterexamples in Topology}.
\newblock Springer-Verlag New York, 2 edition, 1978.

\bibitem{Robotics}
C.~W. Wampler, A.~Morgan, and A.~J. Sommese.
\newblock Complete solution of the nine-point path synthesis problem for
  four-bar linkages.
\newblock {\em J. Mech. Des.}, 114:153--159, 1992.

\bibitem{Ziegler_book}
G.~M. Ziegler.
\newblock {\em Lectures on Polytopes}.
\newblock Springer, 2007.

\end{thebibliography}

}

\end{document}